\documentclass[12pt,reqno,a4paper, english]{amsart}

\usepackage[top=3.53cm,bottom=2.7cm,left=2.3cm,right=2.3cm]{geometry}

\usepackage{amsmath, amsthm, amsfonts, amssymb,enumerate}
\usepackage{amsthm} % For theorem without numbering
\usepackage[mathscr]{euscript}  %For \mathscr{}
\usepackage{mathtools}%makes typesetting equations easy, and mathtools makes those equations beautiful
\usepackage{slashed}% for slashing operators
\usepackage{dsfont} % pour 1 avec double bar
\usepackage{upgreek}% to use \uppsi,\upphi
\usepackage{nccmath}% to use small size of big fraction. For example \mfrac
\usepackage{stmaryrd}% cross out the arrow of convergence

\usepackage[dvipsnames]{xcolor} % To cancel the error that appears while using tikz with color
\usepackage{tikz}% To draw
\usepackage{microtype} %Alignement des fin de phrases
\usepackage{comment} % instead of the ctrl+right slash /
\usepackage[titletoc,title]{appendix}
\usepackage[font=small]{caption} %Caption--Legende d'une figure
\usepackage{systeme} % Exmaple : \system[x,y,z]{2x+4y+3z=2,x+z=1,y+z=2}

\usepackage{etoolbox}

\usepackage{hyperref}
\hypersetup{colorlinks={true},linkcolor={awesome},citecolor={greenBleu},urlcolor={britishracinggreen}}
\definecolor{awesome}{rgb}{1, 0.1, 0}
\definecolor{amber}{rgb}{1.0, 0.75, 0.0}
\definecolor{greenBleu}{rgb}{0.0, 0.6, 0.5}
\definecolor{britishracinggreen}{rgb}{0.0, 0.26, 0.15}
\definecolor{caputmortuum}{rgb}{0.7, 0.55, 0.5}

\makeatletter
\newenvironment{figurehere}
{\def\@captype{figure}}
{}
\makeatother

\theoremstyle{plain}
\newtheorem{theorem}{Theorem}[section]
\newtheorem*{theorem*}{Theorem}
\newtheorem{prop}[theorem]{Proposition}

\newtheorem{corollary}[theorem]{Corollary}
\newtheorem*{corollary*}{Corollary}

\newtheorem{lemma}[theorem]{Lemma}

\theoremstyle{definition}
\newtheorem{defi}[theorem]{Definition}

\theoremstyle{remark}
\newtheorem{Rq}{Remark}[section]

\numberwithin{equation}{section}

\DeclareMathOperator{\eee}{e}

\DeclareMathOperator{\Id}{Id}

\newcommand{\be}{\begin{equation}}
	\newcommand{\ee}{\end{equation}}
\newcommand{\bes}{\begin{equation*}}
	\newcommand{\ees}{\end{equation*}}

\newcommand{\mrm}[1]{\mathrm{#1}}

\newcommand{\R}{\mathbb{R}} 
\newcommand{\C}{\mathbb{C}} 
\newcommand{\N}{\mathbb{N}} 
\newcommand{\D}{\mathbb D} 
\newcommand{\T}{\mathbb{T}}
\newcommand{\Z}{\mathbb{Z}}

\newcommand{\Nzero}{\mathbb{N}_{\geq0}}

\newcommand{\Sh}{S}

\newcommand{\qlq}{\forall}
\newcommand{\e}{\exists}

\newcommand{\p}{\partial}

\newcommand{\eps}{\varepsilon}
\newcommand{\la}{\lambda}

\newcommand{\Ltwo}{L^2_+(\T)}
\newcommand{\Htwo}{H^2_+(\mathbb{T})}

\newcommand{\lracc}[1]{\left\{#1\right\}} %\set
\newcommand{\uzeroeps}{u_0^\eps}
\newcommand{\uzeroepsbar}{\bar{u}_0^\eps}
\newcommand{\ueps}{u^\eps}

\newcommand{\fnzero}{f_n^{\,0}}

\newcommand{\fnzeroeps}{f_n^{\,\eps,0}}

\newcommand{\fnepst}{f_n^{\,\eps,t}}

\newcommand{\fnt}{f_n^{\,t}}

\newcommand{\B}[1]{\mathcal{B}_{L^2_+}(#1)}

\newcommand{\vertiii}[1]{{\vert\kern-0.25ex\vert\kern-0.25ex\vert #1 
		\vert\kern-0.25ex\vert\kern-0.25ex\vert}}% norm : |||.|||

\newcommand{\va}[1]{\lvert#1\rvert}

\newcommand{\Ldeux}[1]{\left\| {#1} \right\|_{L^2(\T)}} 

\newcommand{\ps}[2]{  \left\langle #1\,|\, #2  \right\rangle }

\newcommand{\fr}{\widehat}

\newcommand{\en[1]}{\mathrm{e}^{i{#1}x}} % exp(inx)--Fourier basis

\let\oldtocsection=\tocsection

\let\oldtocsubsection=\tocsubsection

\let\oldtocsubsubsection=\tocsubsubsection

\renewcommand{\tocsection}[2]{\hspace{0em}\oldtocsection{#1}{#2}}
\renewcommand{\tocsubsection}[2]{\hspace{2em}\oldtocsubsection{#1}{#2}}
\renewcommand{\tocsubsubsection}[2]{\hspace{4.7em}\oldtocsubsubsection{#1}{#2}}

\begin{document}

	%%%%%%%%%%%%%%%%%%%%%%%%%%%%%%%%%%%%%%%%%%%%%
	%%%%%%%%%%%%%%% Title details %%%%%%%%%%%%%%%
	%%%%%%%%%%%%%%%%%%%%%%%%%%%%%%%%%%%%%%%%%%%%%

	\title[Global well--posedness of CS-DNLS equation]{On the global well--posedness of the Calogero--Sutherland Derivative nonlinear Schrödinger equation 
		% in $H^s_+(\T)\,$, $s\geq0$ 
	}
	
	% \title[Global well--posedness of CS-DNLS in $H^{\frac{k}{2}}_+(\T)$\,, $k\in \Nzero$]{On global well--posedness results for the Calogero--Sutherland Derivative non--linear Schrödinger equation}

	\author{Rana Badreddine}
	\address{Universit\'e Paris-Saclay, Laboratoire de math\'ematiques d'Orsay, UMR 8628 du CNRS, B\^atiment 307, 91405 Orsay Cedex, France}
	\email{\href{mailto: rana.badreddine@universite-paris-saclay.fr}{rana.badreddine@universite-paris-saclay.fr} }

	\date{\today}

	%%%%%%%%%%%%%%%%%%%%%%%%%%%%%%%%%%%%%%%%%%%%%
	%%%%%%%%%%%%%%%%% Abstract %%%%%%%%%%%%%%%%%%
	%%%%%%%%%%%%%%%%%%%%%%%%%%%%%%%%%%%%%%%%%%%%%
	\begin{abstract}
		We consider the Calogero--Sutherland derivative nonlinear Schrödinger equation in the focusing (with sign $+$) and defocusing case (with sign $-$)
		$$
		i\partial_tu+\partial_x^2u\,\pm\,\frac2i\,\p_x\Pi(|u|^2)u=0\,,\qquad (t,x)\in\R\times\T\,,
		$$
		where  $\Pi$ is the Szeg\H{o} projector $\Pi\left(\sum_{n\in \Z}\fr{u}(n)\eee^{inx}\right)=\sum_{n\geq 0 }\fr{u}(n)\eee^{inx}$\,. Thanks to a Lax pair formulation, we derive the \textit{explicit solution} to this equation. Furthermore, we prove the \textit{global well--posedness} for this $L^2$--critical equation in all the Hardy Sobolev spaces $H^s_+(\T)\,,$ $s\geq0\,,$ with small $L^2$--initial data in the focusing case, and for arbitrarily $L^2$--data in the defocusing case. In addition, we establish the relative compactness of the trajectories in all $H^s_+(\T)\,,$ $s\geq0\,.$ 
	\end{abstract}

	% \ \vskip -1cm  \hrule \vskip 1cm \vspace{-8pt}
	%  \maketitle 
	% { \textwidth=4cm \hrule}
	
	\maketitle
	\tableofcontents
	
	\section{Introduction}
	
	% \vskip0.2cm
	This paper aims to prove the global well--posedness 
	%  in $H^{\frac{k}{2}}_+(\T)\,,$ $k\in \N_{\geq 0}$ of 
	for the Calogero--Sutherland derivative nonlinear
	Schrödinger equations  on the torus \big($x\in\T := \R/(2\pi \Z)$\big ) :
	\be
	\label{CS}
	\tag{CS}
	\begin{cases}
		i\p_tu+\p_x^2u\pm 2D_+(\va{u}^2)u=0\,, \\
		u(t=0,x)\,=\,u_0\,, \hskip0.9cm x\in \T\,,
	\end{cases}
	\ee
	for small $L^2$-initial data $u_0\,$ in the focusing case (with sign $+$)\,, and for arbitrarily $L^2$--initial data in the defocusing case (with sign $-$).
	% typically $\|u_0\|_{L^2}<1\,,$
	The operator  $D_+$ in the nonlinear term of \eqref{CS} denotes $D\Pi\,$, where
	$\;D=-i\p_x\,,$ and $\Pi$ is the Szeg\H{o} projector  acting on $L^2(\T)$  as
	\be\label{szego proj}
	\Pi\left(\sum_{n\in \Z}\fr{u}(n)\eee^{inx}\right):=\sum_{n\geq0 }\fr{u}(n)\eee^{inx}\,,
	\ee
	with value onto the Hardy space
	\be\label{L2+}
	L^2_+(\T):=\lracc{u\in L^2(\T)\,|\ \fr{u}(n)=0\,,\,\qlq n\in \Z_{\leq -1}}\equiv\, \Pi(L^2(\T))\,.
	\ee
	We equip $\Ltwo$ with the standard inner product of $L^2(\T)$\,, $\ps{u}{v}=\int_0^{2\pi}u \bar v \,\frac{dx}{2\pi }\,.$
	Our interest focuses on studying this equation with
	an unknown function 
	$u$ 
	taken
	%  as an element of 
	in the Hardy space of the torus, with a certain regularity.
	Thus, we denote by $H^s_+(\T)\,,$ the subspace of the Sobolev space $H^s(\T)\,,$ defined as 
	\be\label{Hs+}
	H^s_+(\T):=H^s(\T)\cap L^2_+(\T)\,, \qquad s\geq 0\,,
	\ee
	and equipped with the Sobolev norm
	$$
	\|u\|_{H^s}=\|\langle D \rangle^{s}u\|_{L^2}\,,
	\qquad 
	\langle D \rangle^s= (1+|D|^{2})^{s/2}\,.
	$$
	% \vskip0.25cm
 
	In Physics, this dynamical \eqref{CS}--equation  is derived from the classical Calogero--Sutherland--Moser system (or Toda system) introduced in the end sixties--early seventies \cite{Ca69,Ca71,Su71,Su75}. This physical model corresponds to a $N$--body problem describing the pairwise interactions of $N$ identical particles.
	% This interaction can be attractive or repulsive leading to the focusing or the defocusing cases.
	Abanov--Bettelheim--Wiegmann show in 
	%  their paper 
	\cite{Abanov09} 
	that taking the thermodynamic limit of such a model, and applying a change of variables leads to the 
	% focusing 
	\eqref{CS}--equation. 
  In Mathematics, this equation has recently been studied on the real line ($x\in\R$) by \cite{GL22}\,, who referred to the equation as the Calogero--Moser derivative NLS equation. 
  The transition of nomenclature to the Calogero--Sutherland DNLS equation in the periodic setting ($x\in\T$) is connected to the physicist Sutherland, who has studied the Calogero--Sutherland--Moser system in the case where the $N$ particles lie on the circle and interact with an inverse sin-square potential (trigonometric--type potential). 
Besides, one can obtain the \eqref{CS}--equation formally as a limit of the \textit{intermediate nonlinear Schrödinger equation} introduced by Pelinovsky \cite{Pe95}\,,
	\be\label{INS}\tag{INS}
	i\p_tu = \p_x^2u + (i - T )\p_x(\lvert u\rvert^2)\,u\,,
	\ee
	where  
	$T$ is the integral operator
	$$
	Tu(t,x)=\frac{1}{2\delta} \,\text{p.v.}\int_{-\infty}^{+\infty}\coth\Big(\frac{\pi(x-y)}{2\delta}\Big) u(t,y)\, dy\,,
	$$
	by taking  $\delta\to \infty$\,.  The complex function $u$ in \eqref{INS} represents the envelope of the fluid, and $\delta$ denotes its total depth. By passing to the limit $\delta\to\infty\,,$ one obtains the same equation as \eqref{INS} but with the Hilbert transform 
	$$
	Hu(t,x)=\frac{1}{\pi}\,\text{p.v.}\int_{-\infty}^{+\infty}\frac{u(t,y)}{x-y}\, dy\,,
	$$
	instead of $T$ \cite{Pe95}.  And since the Szeg\H{o} projector $ \Pi_\T=\frac12(\Id+iH+\ps{\cdot}{1})\,,$ then the \eqref{CS}--equation can also be interpreted as a model describing  
	the interfacial 
 wave packets in a deep stratified fluid.

	\vskip0.25cm
	It turns out that the Calogero--Sutherland DNLS equation is completely integrable. 
	Thus, what does the word ``integrability'' mean? 
	% Well, a number of definitions have been raised according to the perspective developed by the different schools.
	In line with the different perspectives developed by various schools,
	%to define the word ``integrability'', 
	a number of definitions have been raised.
	If the word ``integrable system'' means for some researchers the existence of {action--angle variables}, a coordinate system in which the equation is completely solvable by quadratures, others would say that it refers to the existence of a Lax operator associated with the equation, and satisfying the isospectral property\footnote{~See Remark~\ref{val propres conservées}.}. 
	However, a common facet of all these definitions is the presence of infinitely independent \textit{integrals of motion}, or what we can also call \textit{conservation laws}.
	Naturally, this infinite number of conservation laws plays a crucial role in proving some global well--posedness results. 
	\vskip0.25cm
	In our case, G\'erard--Lenzmann derived in \cite[Lemma 5.1]{GL22}\,,
	for $u$ sufficient regular,
	a Lax operator so that 
	the focusing Calogero--Sutherland DNLS equation \eqref{CS}$^+$ enjoys a Lax pair formulation on the real line $\R\,.$ i.e., for any $u\in H^s_+(\R)$ with $s$ sufficiently large, 
	% for instance $u\in H^2_+(\R)$\,, 
	there exist two operators $(L_u,B_u)$ such that the Lax equation 
	\be\label{Lax equation}
	\frac{dL_u}{dt}=[B_u\,,L_u]\,,\qquad\quad [B_u\,,L_u]:= B_uL_u -L_uB_u\,,
	\ee
	is satisfied with
	\be\label{Lax operators}
	\hskip1.4cm
	L_u=D-T_uT_{\bar{u}}\,,
	%    \quad\text{ and }
	\quad \hskip0.9cm
	B_u=T_{u}T_{\partial_x\bar{u}}-T_{\partial_x{u}}T_{\bar{u}} +i(T_{u}T_{\bar{u}})^2\,.
	\ee
	The operator $T_u$ is the Toeplitz operator of symbol $u$\,, and is defined for any $u\in L^\infty$ by
	\be\label{Toeplitz operator}
	T_{u}f=\Pi(uf)\,, \qquad \qlq f\in L^2_+\,,
	\ee
	where $\Pi$ is the Szeg\H{o} projector given in \eqref{szego proj}\,.
	In what follows, we check that this Lax equation holds true on the torus $\T$ by retrieving the same Lax operators  $(L_u,B_u)$ as on the real line\,.
	% ,  even though a complex function $f$ is decomposed as $f=\Pi f+\overline{\Pi \overline{f} }$ on $\R\,,$  while $f=\Pi f+\overline{\Pi \overline{f}}-\ps{f}{1}\,$ on the torus $\T\,.$ 
	And, as expected, through this Lax formalism, we derive infinite conservation laws $\ps{(L_u+\la)^s 1}{1}\,,$ $\la>\!>0\,,$ $s\geq 0 $\,, in order to control the growth of the Sobolev norms $\|u(t)\|_{\dot H^{s}}$ uniformly for all $t\in\R\,.$\footnote{~In particular, one can see that the usual conservation laws : the average $\ps{1}{u}$\,, and the $L^2$ norm $\|u\|_{L^2}$ are conserved for $s=1$ and $2$\,, since by definition of $L_u=D-T_uT_{\bar{u}}$ we have $L_u1=-\ps{1}{u}u$\,.} 
 
	\vskip0.25cm
	Observe, the Calogero--Sutherland DNLS equation is invariant under the scaling 
	\be\label{scaling invariance}
	u(t,x)\longmapsto \la^{1/2}u(\la^2t,\la x)\ ,\qquad \la\in\R\,,\; 
	(t,x)\in {I}\times\R\,.
	\ee
	This suggests the $L^2$--criticality of \eqref{CS} on $\R$ as well as on $\T$\,. In \cite[Theorem~2.1]{GL22}\,,  the local well--posedness of the \eqref{CS} equation was achieved in $H^s_+(\R)$ for $s>\frac12$\, by following the analysis of \cite{deMP10}. In particular, for $s>
	\frac32$\,, G\'erard--Lenzmann \cite[Proposition~2.1]{GL22} used iterative schemes of Kato's type and energy estimates to derive the local well--posedness in $H^{s}_+(\R)\,$ \cite{Sa79}\,.
	On $\T\,,$ the same proof of iterative schemes holds, and we deduce the local well--posedness in $H^s_+(\T)$ for $s>\frac32$\,.
	Therefore, we denote by $\mathcal S^+(t)$ the flow of the focusing Calogero--Sutherland DNLS equation \eqref{CS}$^+$  and by $\mathcal S^-(t)$ the flow of the defocusing equation \eqref{CS}$^-$ : for all $s>\frac32\;,$ $t\in\mrm{I}_{\max}$\,, 
	\be\label{flot}
	\begin{array}{cccc}
		\mathcal S^{\pm}(t)&:H^{s}_+(\T)&\longrightarrow& H^{s}_+(\T)\\
		&u_0&\longmapsto&u(t)
	\end{array}\,,
	\ee
	where $\mrm{I}_{\max}$ denotes the maximal interval of the existence of the solution.
	~\\
	
	\subsection{Main results}
	\hfill
 \\
 \textbf{Some notation.}
	In the sequel, we denote for any nonnegative integer $a\,,$  by $\N_{\geq a}$ the subset of $\Z$ given by $\lracc{k\in \Z \mid k\geq a}$\,. Moreover, we denote by $\mathcal{B}_{L^2_+}(r)$ the open ball of $\Ltwo$ centered at the origin, with radius $r>0$\,.
	~\\
	
	The goal of the paper is to prove the global well--posedness of the $L^2$--critical equation \eqref{CS} in all $H^s_+(\T)$\,, $s\geq 0\,.$
	As a starting point, we state the results for the more challenging equation, the focusing Calogero--Sutherland DNLS equation  
	\be\label{CS+}\tag{CS$^+$}
	i\p_tu+\p_x^2u+ 2D_+(\va{u}^2)u=0\,,
	\ee
	then, we present the results for the defocusing case~\footnote{~
		We refer to the introduction of Weinstein \cite{We15} for a mathematical and physical meaning of the terms focusing and defocusing for any dispersive equation.
	}
	\be\label{CS-defocusing}\tag{CS$^-$}
	i\p_tu+\p_x^2u-2D_+(\va{u}^2)u=0\,.
	\ee
	
	\begin{theorem}\label{GWP k geq 4}
		For all $s>\frac32\,,$ the Calogero--Sutherland DNLS focusing equation \eqref{CS+} is globally well--posed in 
		$H^s_+(\T)\cap \mathcal{B}_{L^2_+}(1)$\,.
		% initial data $u_0\in H^2_+(\T)$\,, .
		Moreover,  the following a priori bound holds, 
		\[
		\sup_{t\in\R}\|u(t)\|_{H^s}\;\leq C\;,
		\]
		where $C(\|u_0\|_{H^s})$ is a positive constant.
	\end{theorem}
	
	\begin{Rq}
		% As we will see, t
		The restriction of smallness on the $L^2$--norm of the initial data, 
		namely $\|u_0\|_{L^2}<~1$\,, 
		appears after applying a sharp inequality (Lemma~\ref{inegality Pi(u bar h)}) in order to control the growth of the Sobolev norms $\|u(t)\|_{\dot{H}^s}$\,, $\,s\geq 0\,,$ by the conservation laws. 
 This sharp inequality is an equality when we take for example 
  \[
  u_0(x)=\frac{\sqrt{1-\va{q}^2}}{1-q\eee^{ix}}\,, \quad q\in \D\,,
  \]
  which correspond to the profile of a \eqref{CS+}'s traveling wave of $L^2$--norm $\|u_0\|_{L^2}=~1$ \cite{Ba23}.
  % (see the third point of Remark~\ref{Hankel op})
		More details for an eventual way to avoid this condition are presented in Section~\ref{Remarks}\,, but so far it is still an open problem.
	\end{Rq}
	
	\vskip0.25cm
	As a second step, we focus on the main point of this paper : how the flow $\mathcal{S}^+(t)$ defined globally on $H^2_+(\T)$ for $u_0\in \B{1}\,,$ can be extended to less regularity spaces for instance $L^2_+(\T)\,$?  Recall, as noted in \eqref{scaling invariance}\,, the Calogero--Sutherland DNLS equation is $L^2$--critical. 
	Based on the previous Theorem, and under the notation $u^\eps(t)=\mathcal S^+(t)u_0^\eps\;,$ $\eps>0\,,$ we state the following result.
	\begin{theorem}\label{extension du flot a L2}
		Let $u_0\in \B{1}\,.$
		There exists a unique potential $u\in\mathcal C(\R; L^2_+(\T))$ such that
 for any sequence $(u_0^\varepsilon)\subseteq H^2_+(\T)$ where
 % approximating $u_0$ in $L^2\,,$
 $
    \|u_0^\eps-u_0\|_{L^2}\underset{\eps\to0}{\longrightarrow} 0\,, 
$
the following convergence holds : for all~$T>0\,,$
$$
    \sup_{t\in[-T,T]}\|u^\eps(t)-u(t)\|_{L^2}\to 0\,, \quad \eps\to 0\,.
$$
  Moreover, the $L^2$--norm of the limit potential $u$ is  conserved 
		\be\label{egalite norme L2 intro}
		\|u(t)\|_{L^2}=\|u_0\|_{L^2}\,,\quad \qlq t\in \R.
		\ee
	\end{theorem}
	As a consequence, Theorem~\ref{extension du flot a L2} leads to the global well--posedness of the \eqref{CS+} problem in $L^2_+(\T)$ in the following sense : There exists a unique continuous extension of the flow defined on $H^2_+(\T)\,,$ to $\Ltwo\,,$ generating a unique continuous map  
	$$
	u_0\in\mathcal{B}_{L^2_+}(1)\longmapsto u\in\mathcal{C}(\R,\Ltwo)\,.
	$$
	The key ingredient of the proof is to obtain $H^\frac12$ bounds (inequality \eqref{borne to fn-uzeroeps-H1/2}) on the eigenfunctions of the Lax operator $L_{u^\eps}$\,, which also constitute an orthonormal basis of $L^2_+(\T)\,$. 
 Therefore, we deduce the strong convergence of these eigenfunctions in $L^2$\,. 
 Finally, using Parseval's identity, we infer~\eqref{egalite norme L2 intro}\,.
	\vskip0.25cm
	We also need to emphasize the important aspect of the uniqueness of the limit potential $u(t)\,,$ obtained independently of the choice of the sequence $(u_0^\eps)$ that approximates $u_0\in L^2_+(\T)\,$.
	For this purpose, we derive in Proposition~\ref{The inverse dynamical formula}\,, \textbf{an explicit formula of the solution} of the focusing \eqref{CS+} equation. Thus, for any initial
	data $u_0$\,,  the solution of the \eqref{CS+} focusing equation is given by 
	% for all $z\in\D:=\lracc{z\in\C\,;\,\va{z}<1},$ 
	\be\label{explicit formula intro}
	u(t,z)=\ps{(\Id-z\eee^{-it}\eee^{-2itL_{u_0}}S^*)^{-1}\,u_0}{1}\,,\qquad z\in\D:=\lracc{\va{z}<1}\,,
	\ee
	where $S^*$ denotes the adjoint of the Shift operator $S:h\mapsto zh$ in $L^2_+(\T)\,,$  and $L_{u_0}$ is the Lax operator at $t=0\,.$
	We underline two important facts about \eqref{explicit formula intro} :
	\begin{enumerate}[I.]
		\item \label{I} First, this inversion dynamical formula defined inside the open unit disc consists an explicit solution for the nonlinear PDE \eqref{CS+}\,. This is not the first time that an explicit solution occurs while dealing with nonlinear integrable PDEs.
		Indeed, G\'erard--Grellier derived in \cite{GG15} an explicit solution for the Szeg\H{o} equation, and recently G\'erard also prove in \cite{Ger22} that the Benjamin--Ono equation has an explicit solution on $\R$ and on $\T\,.$ 
		The common point to all these dynamical explicit formulas is that they all rely closely
		% and are closely related to the
		on the structure of the Lax operators induced by these equations.
		\vskip0.1cm
		\item Beyond the fact that we have an explicit solution, this formula stresses out that the dynamics of the \eqref{CS+} equation
		are encoded by the Lax operator $L_{u_0}$\,, suggesting thus, that the so--called \textit{actions--angles variables}
		must be related to the spectral elements of the Lax operators $L_{u}\,$.
	\end{enumerate}
	
	\vskip0.25cm
	
	In view of  Theorem~\ref{extension du flot a L2}\,, we state the third result.
	\begin{corollary}
		\label{GWP k geq 0} For all $0\leq s\leq\frac32\,,$
		the Calogero--Sutherland DNLS focusing equation \eqref{CS+} is globally well--posed in $H^s_+(\T)\cap \mathcal{B}_{L^2_+}(1)$\,.   Moreover,  the following a--priori bound holds, 
		\[
		\sup_{t\in\R}\|u(t)\|_{H^s}\;\leq C\;,
		\]
		where $C=C(\|u_0\|_{H^s})>0$ is a positive constant.
	\end{corollary}
	
	\begin{Rq}\label{Rq GWP}
		There is a subtlety hidden in the words of ``globally well--posed'' in the last statement. In fact, 
		it is important to distinguish here the two different aspects of global well--posedness. 
		First, we have the classical definition of GWP used in Theorem~\ref{GWP k geq 4}\,: for any $u_0\in H^s_+$ there exists a unique solution $u$  defined on $\R$ with value in $H^s_+$\,, such that $u$ depends continuously  on the initial data $u_0$ as a map $u_0\in H^s_+\mapsto u\in \mathcal{C}(\R,H^s_+)\,.$
		The second definition is the one described in Theorem~\ref{extension du flot a L2} in the sense : we suppose that the equation is defined at least in the distribution sense, then we extend the flow defined on high regularity spaces to low regularity spaces through continuous extension.
		
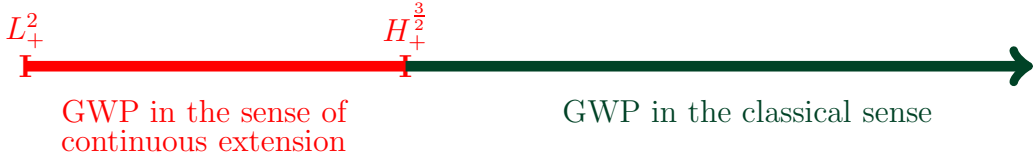
\begin{figurehere}
			\begin{center}
				\begin{tikzpicture}
					\draw[->] (-0,0)--(13.25,0);
					
					\draw[red,-, line width=4] (-0,0)--(5,0);
					\draw[red] (0,0) node{\bf{I}} node[above=2.5] {$L^2_+$};
					
					\draw[red] (2.35,-0.3) node[below] {GWP in the sense of };
					\draw[red] (2.35,-0.7) node[below] { continuous extension};

					\draw[red] (5,0) node{{\bf I}} node[above=2]{$H^\frac32_+$};
					
					\draw[britishracinggreen,->, line width=4,](5,-0)--(13.25,-0);
					\draw[britishracinggreen] (9.5,-0.3) node[below] {GWP in the classical sense};
					
				\end{tikzpicture}
			\end{center}
		\end{figurehere}
		In this corollary, the global well--posedness is in the sense used in Theorem~\ref{extension du flot a L2}\,.
		This will become clearer once the proof is established (see Section~\ref{Proof of th 3 and 4}). We also expect that, following arguments in \cite{deMP10}\,, one can go down for the global well--posedness in the classical sense to $H^s_+(\T)$ with $s>\frac12\,.$
	\end{Rq}    
	
	\vskip0.25cm
	Beyond the global well--posedness results on the Cauchy Problem of \eqref{CS+}\,, we are interested in some qualitative properties about the flow $\mathcal S^+(t)$ of this equation.
	
	\begin{theorem}\label{Compact}
		Given an initial data $u_0\in \B{1}\cap H^s_+(\T)$\,, $s\geq 0\,,$ the orbit of the solution $\lracc{\mathcal{S}^+(t)u_0\,;\,t\in\R}$ is relatively compact in $H^s_+(\T)\,.$
	\end{theorem}
	\vskip0.2cm
	% \begin{center}
	% 	***
	% \end{center}
	% \vskip0.2cm
	\noindent
	\textbf{The defocusing equation \eqref{CS-defocusing}.} Moving now to the defocusing case  of the Calogero--Sutherland DNLS equation, this latter equation  enjoys also a Lax pair structure : for any $u(t)\in H^s_+\;,$ with $s$ large enough, there exist two operators 
	$$
	\tilde{L}_u=D \,\textcolor{red}{+}\, T_uT_{\overline u}\,, \qquad\tilde{B}_u=\,\textcolor{red}{-}T_uT_{\partial_x\overline u}\,\textcolor{red}{+}\,T_{\partial_xu}T_{\overline u} +i(T_uT_{\overline u})^2\,,
	$$
	satisfying  the Lax equation 
	$$
	\frac{d\tilde{L}_u}{dt}=[\tilde{B}_u,\tilde{L}_u]\,.
	$$ 
	Therefore, using the same methods as on the focusing case, we prove that the conservation laws $\langle(\tilde{L}_u+\lambda)^s1\mid 1\rangle$\,, $s\geq 0\,,$ $\lambda>0\,,$ controls uniformly the growth of the Sobolev norms 
	without requiring any additional condition on the initial data.
	As a consequence, we obtain similar results in the defocusing case as in the focusing case, \textit{regardless of how large the initial data is in $L^2$\,.} 
	To summarize, we have the following.
	
	\begin{theorem}\label{th defocusing case}
		The Calogero--Sutherland DNLS defocusing equation \eqref{CS-defocusing} is globally well--posed in
		$H^s_+(\T)$ for any $s\geq 0$ in the sense of Remark~\ref{Rq GWP}\,. 
		In addition,
		for all $u_0\in H^s_+(\T)\,,$
		\[
		u(t,z)=\ps{(\Id-z\eee^{-it}\eee^{-2it\tilde{L}_{u_0}}S^*)^{-1}\,u_0}{1}\,, 
		\]
		is the solution to the \eqref{CS-defocusing}--defocusing equation.
		Furthermore, the trajectories 
		\[
		\lracc{\mathcal{S}^-(t)u_0\,;\,t\in\R}
		\]
		are relatively compact in $H^s_+(\T)\,.$
	\end{theorem}
	% \begin{center}
	% 	***
	% \end{center}

	\subsection{Other related equations}
	As explained in \cite{GL22}, the Calogero--Sutherland DNLS equation \eqref{CS} can be seen as mass critical version
	of the Benjamin–Ono equation. We refer to \cite{GK21,GKT20} for a deep study of this latter equation on the torus. Of course,  the Calogero--Sutherland DNLS equation \eqref{CS} is also considered as part of the nonlinear Schrödinger's family.
	Several authors have been interested in different types of NLS--equations over the years. Some of these equations are classified and presented in \cite{Bo99}\,. Maybe the most closely related to the \eqref{CS}--equation are :
	\begin{enumerate}[(i)]
		\item \textbf{Cubic NLS equation.}
		% \footnote{~which physically correspond to \eqref{INS} where the total depth $\delta\to0\,.$}
		\be
		\label{NLS-Cubic}
		\tag{NLS-cubic}
		i\p_tu+\p_x^2u \pm \va{u}^2u=0\,,
		\ee
		which is considered as one of the simplest PDE enjoying complete integrable properties.  Zakharov--Shabat have studied this equation in \cite{ZS72} using inverse scattering method. 
		Moreover,  global well--posedness results in $L^2(\T)$ are presented in Bourgain \cite{Bo93} after he introduced the $X^{s,b}$--spaces. 
		His proof relies on establishing $L^4(\T)$--Strichartz estimates and using $L^2$--conservation norm. Actually,
		this result of $L^2$--well--posedness is known to be sharp, and it is illustrated by various types of ill--posedness results below the regularity $L^2(\T)$\,.
		% Indeed, they have been various types of ill--posedness results below the regularity $L^2$\,. 
		Indeed, Burq--G\'erard--Tzvetkov  proved in \cite{BGT02} that the flow map of \eqref{NLS-Cubic} fails to be uniformly continuous for Sobolev regularity below $L^2\,.$ Christ--Colliander--Tao \cite{CCT03} and Molinet \cite{Mo09} showed the discontinuity of the map solution in $H^s(\T)$ for $s<0\,.$
		% this $L^2$-wellposedness is sharp and they illustrate different types of ill--posedness
		
		For a deep study of \eqref{NLS-Cubic} using integrable tools, Birkhoff normal form, and some applications, we refer to Kappeler--Lohrmann--Topalov--Zung \cite{KLTZ17}\,, Gr\'ebert--Kappeler \cite{GK14} and Kappeler--Schaad--Topalov \cite{KST17}.
		%  how provide a deep study of this equation using Birkhoff normal form. 
		% We also refer to \cite{KST17} for other studies in this equation.
		For a study on the line $\R$, we cite \cite{HKV20}. More references are also provided in \cite{OS12}\,.
		\item \textbf{DNLS equation.}
		\be
		\label{DNLS}
		\tag{DNLS}
		i\p_tu+\p_x^2u +\pm i \p_x\left(\va{u}^2u\right)=0\,,
		\ee
		which is also an integrable equation enjoying infinite conservation laws \cite{KN78}. Using the I-method,  Win proved in \cite{Wi10} the global well--posed of \eqref{DNLS}--equation in $H^s(\T)$\,, $s>\frac12$ for small data in $L^2(\T)\,.$ More recently, Klaus--Schippa \cite{KS22} presented law regularity a priori estimates of $\|u\|_{H^s}$ for $0<s<\frac12$ upon small $L^2$--norm, where $u\in\mathcal{C}^\infty(\R,\mathcal{S}(\T))$ 
		and $\mathcal{S}(\T)$ denotes the Schwartz space. Actually, they proved the a priori estimates
		$$
		\sup_{t\in\R}\|u(t)\|_{B^s_{r,2}}\lesssim \|u(0)\|_{B^s_{r,2}}
		$$ in any Besov space $B^s_{r,2}\;,$ with $r\in[1, \infty]$ and $0<s<\frac12$\,. For a study on the line $\R$, we cite \cite{JLPS20,BP22, BLP21,KNV21,HKV21,HKNV22}\,.
		% On $\R$\,, better results were achieved in 
	\end{enumerate}
	
	\subsection{Outline of the paper} 
	The paper is organized as follows. 
	\\
	In Section~\ref{The Lax pair formalism}\,, we discuss some properties about the Lax operators of the  Calogero--Sutherland DNLS focusing equation \eqref{CS+}\,. We derive the explicit formula of the solution of \eqref{CS+} in the first subsection~\ref{subsect: explicit formula}\,. Then, we prove in the second subsection~\ref{GWP-in-H2-if-u0<1}\,, the global well--posedness of the \eqref{CS+} problem in $H^{s}_+(\T)$ for any $s>\frac32\;$. 
	
	\noindent
	In {Section~\ref{section: Extension du flot}}\,, 
	% the aim is to
	we extend the flow $\mathcal S^+(t)$ of \eqref{CS+} continuously from $\Htwo$ to $\Ltwo\equiv H^0_+(\T)$\,. To this end, we use an approximation method, and we characterize in the {first subsection}~\ref{uniqueness of the limit} the limit potential $u(t)$ for all $t\in\R$\,.
	Then, in the second subsection~\ref{Conservation of L2 mass}, we make sure that the lack of compactness in $L^2_+(\T)$ do not occur while passing to the limit from $H^2_+(\T)$ to $\Ltwo\,.$
	In the same subsection, we derive an orthonormal basis of $\Ltwo$ where the coordinates of the solution $u(t)$ have nice evolution in this basis. This evolution suggests that the so--called ``Birkhoff coordinates'' are the coordinates of $u(t)$ in this basis.
	
	\noindent
	After that,
	we deal in {Section}~\ref{Proof of th 3 and 4} with the problem of global well--posedness of \eqref{CS+} in $H^{s}_+(\T)$ for $0<s\leq \frac32\,$. Moreover, we address the property of relative compactness of the orbits of \eqref{CS+} in $H^{s}_+(\T)$\,, $s\geq 0$\,.
	
	\noindent
	Moving to Section~\ref{defocusing}\,, we present the Lax pair for the defocusing Calogero--Sutherland DNLS equation \eqref{CS-} and we state the analogous results of \eqref{CS+} in the case of \eqref{CS-}\,.
	
	\noindent
	Finally, in Section~\ref{Remarks}\,, we discuss some remarks and open problems related to this equation.

	\section*{Acknowledges.} The author would like to thank warmly her Ph.D. advisor Patrick G\'erard for his rich discussions and comments on this paper. Additionally, she expresses her appreciation to the anonymous referee for the thorough review.

	%%%%%%%%%%%%%%%%%%%%%%%%%%%%%%%%%%%%%%%%%%%%%
	%%%%%%%%%%%%%%% Lax operator %%%%%%%%%%%%%%%%
	%%%%%%%%%%%%%%%%%%%%%%%%%%%%%%%%%%%%%%%%%%%%%
	\vskip0.5cm
	\section{The Lax pair structure}\label{The Lax pair formalism}
	\vskip0.2cm

	As noted in the introduction, we first check that the Lax pair defined in \eqref{Lax operators} holds the same in the context of the torus $\T$ as on the real line $\R$, even though on the real line $\R\,,$ a complex function $f$ is decomposed as 
	$$
	f=\Pi f+\overline{\Pi \overline{f} }\,, \qquad \fr{\Pi f}(\xi)=\mathds{1}_{\xi>0}\fr{f}(\xi)\,,  \quad \xi\in\R\,,
	$$
	% \footnote{~On $\R$\,, $\ \fr{\Pi f}(\xi)=\mathds{1}_{\xi>0}\fr{f}(\xi)\,,$ where $\fr{f}(\xi)$ denotes the Fourier transform of $f\,.$ And recall on $\T\,,$ the Szeg\H{o} projector is defined by $\Pi\left(\sum_{n\in \Z}\fr{f}(n)\eee^{inx}\right)=\sum_{n\geq 0 }\fr{f}(n)\eee^{inx}$\,. 
	% } 
	while on the torus $\T\,,$
	$$
	f=\Pi f+\overline{\Pi \overline{f}}-\ps{f}{1}\,,\qquad     \Pi\left(\sum_{n\in \Z}\fr{f}(n)\eee^{inx}\right):=\sum_{n\in\Nzero }\fr{f}(n)\eee^{inx}\,.
	$$

	\begin{prop}[The Lax pair] \label{The lax pair prop}
		For any $s>\frac32$\,, let $u\in \mathcal C([-T,T],H^s_+(\T))$ be a solution of  the focusing equation \eqref{CS+}\,.
		Then, there exist two operators 
		$$
		L_u=D \,-\, T_uT_{\overline u}\,, \qquad B_u=\,T_uT_{\partial_x\overline u}\,-\,T_{\partial_xu}T_{\overline u} +i(T_uT_{\overline u})^2
		$$
		satisfying  the Lax equation 
		$$
		\frac{dL_u}{dt}=[B_u,L_u]\,,
		$$ 
		where $T_u$ is the Toeplitz operator defined in \eqref{Toeplitz operator}\,.
	\end{prop}

	\begin{proof} 
		Let $u\in \mathcal C([-T,T],H^s_+(\T))$\,, $s>\frac32\,,$ be a solution of \eqref{CS+} equation. 
		% Our aim is to prove for all $h\in H^1_+(\T)$\,,
		% $$
		%     \frac{dL_u}{dt}(h)= B_uL_uh -L_uB_uh\,.
		% $$
		% \\[0.05cm]
		On the one hand, we have by definition of $L_u\,$ and for all $h\in H^1_+(\T)\,,$
		\begin{align*}
			\frac{dL_u(h)}{dt}
			=&-T_{\p_tu}T_{\bar{u}}(h)-T_uT_{\p_t\overline{u}}(h) 
			\\
			=&-T_{i\partial_x^2u+2u\p_x\Pi(|u|^2)}T_{\bar{u}}h-T_uT_{-i\partial_x^2\bar u+2\bar u\, \p_x\overline{\Pi(|u|^2)}}\,h\,.
		\end{align*}  
		Therefore, since $u$ belongs to the Hardy space,
		\begin{align}\label{dLu/dt}
			\frac{dL_u(h)}{dt} 
			=  i\big[T_uT_{\p_x^2\bar{u}}-T_{\p_x^2u}T_{\bar{u}}\big](h)
			\,-\,
			2u\Big[\p_x\Pi(|u|^2)\cdot\Pi(\bar{u} h)+\Pi\left(\p_x\overline{\Pi(|u|^2)}\cdot\bar u h\right)\Big]\,.
		\end{align}
		\\[0.05cm]
		On the other hand, expanding the commutator $[B_u,L_u](h)=B_uL_uh -L_uB_uh$\,, we obtain
		\begin{multline*}
			\label{commutateur}
			T_{u}T_{\partial_{x}\bar{u}}Dh
			-T_uT_{\partial_{x}\bar{u}}T_uT_{\bar{u}}h
			-T_{\partial_{x}u}T_{\bar{u}}Dh
			+T_{\partial_{x}u}T_{\bar{u}}T_uT_{\bar{u}}h
			+i(T_uT_{\bar{u}})^2Dh
			\\
			-D(T_uT_{\p_x\bar{u}}h)
			+T_{u}T_{\bar{u}}T_{u}T_{\partial_{x}\bar{u}}h
			+D(T_{\partial _{x}u}T_{\bar{u}}h)
			-T_{u}T_{\bar{u}}T_{\partial_{x}u}T_{\bar{u}}h 
			-iD(\left( T_{u}T_{\bar{u}}\right) ^{2}h)\,,
		\end{multline*}
		where by the Leibniz rule, $D(T_uh)=-iT_{\p_xu}h+T_uDh\,,\,$ so that
		\begin{gather*}
			D(T_uT_{\p_x\bar{u}}\; \cdot)=T_{u}T_{\partial_{x}\bar{u}}D-iT_{\partial_{x}u}T_{\partial_{x}\bar{u}}-iT_uT_{\partial_{x}^2\bar{u}}\,,
			\\
			D(T_{\partial _{x}u}T_{\bar{u}}\; \cdot)=T_{\partial_{x}u}T_{\bar{u}}D-iT_{\p_x^2u}T_{\bar{u}}-iT_{\partial_{x}u}T_{\partial_{x}\bar{u}}\,,
			\\
			D(\left( T_{u}T_{\bar{u}}\right) ^{2}\;\cdot)=-i[T_{\p_x u}T_{\bar{u}}T_uT_{\bar{u}}+T_{ u}T_{\p_x\bar{u}}T_uT_{\bar{u}}+T_{ u}T_{\bar{u}}T_{\p_xu}T_{\bar{u}}+T_{u}T_{\bar{u}}T_uT_{\p_x \bar{u}}]+\left( T_{u}T_{\bar{u}}\right) ^{2}D\cdot\;.
		\end{gather*} 
		As a consequence,
		\begin{align}\label{[Bu,Lu]}
			[B_u,L_u](h)=&\,i\,T_uT_{\partial_{x}^2\bar{u}}h-iT_{\p_x^2u}T_{\bar{u}}h-2\left(T_uT_{\bar{u}}T_{\p_xu}T_{\bar{u}}+T_uT_{\p_x\bar{u}}T_uT_{\bar{u}}\right)(h)\notag
			\\
			=&\, i\big[T_uT_{\p_x^2\bar{u}}-T_{\p_x^2u}T_{\bar{u}}\big]-2u\cdot\Pi\left(\p_x|u|^2\,\cdot\,\Pi(\bar{u} h)\right)\,.
		\end{align}
		\\[0.05cm]
		Comparing \eqref{dLu/dt} and \eqref{[Bu,Lu]}\,, it appears that all that remains to be proved is 
		\be\label{decomposition in Pi}
		\Big[\p_x\Pi(|u|^2)\cdot\Pi(\bar{u} h)+\Pi\big(\p_x\overline{\Pi(|u|^2)}\cdot\bar uh\big)\Big]
		=
		\Pi\left(\p_x|u|^2\cdot\Pi(\bar{u}h)\right)\,, \quad h\in H^1_+(\T)\,.
		\ee
		In fact, any complex function $f\in L^2(\T)$ can be decomposed as
		$$f=\Pi f+\overline{\Pi\bar{f}} -\ps{f}{1}\,.$$ In particular, for $f=\bar{u}h$\,, we have $\Pi(\p_x\overline{\Pi(|u|^2)}\cdot\bar uh)$ equal to
		\begin{align*}
			\Pi\left(\p_x\overline{\Pi(|u|^2)}\cdot\Pi(\bar uh)\right)
			+
			\Pi\left(\p_x\overline{\Pi(|u|^2)}\cdot\overline{\Pi( u\bar h)}\right)
			-
			\ps{\bar u h}{1}\Pi(\p_x\overline{\Pi(|u|^2)})\,,
		\end{align*}
		where the last two terms vanishes, since $\Pi$ is an orthogonal projector into the Hardy space. Therefore, the left--hand side of \eqref{decomposition in Pi} coincides with
		\begin{align*}
			\Pi\Big(\p_x\Pi(|u|^2)\cdot\Pi(\bar{u} h)\Big)+\Pi\Big(\p_x\overline{\Pi(|u|^2)}\cdot\Pi(\bar{u} h)\Big)\,,
		\end{align*}
		which is equal to $\Pi\left(\p_x|u|^2\cdot\Pi(\bar{u}h)\right)$ since $\ps{\p_x(\va{u}^2)}{1}=0\,.$
	\end{proof}

	\subsection{The explicit formula of the solution}\label{subsect: explicit formula}
	\vskip0.5cm  
	Using this Lax pair structure, we derive in this subsection the explicit formula of the solution of the focusing Calogero--Sutherland DNLS equation \eqref{CS+}\,.  To this end, we also need the shift operator introduced in the following paragraph.
	
	\vskip0.25cm
 \noindent
	\textbf{Some Preliminaries.}  We recall one of the most important operator on Hardy's space,  \textit{the shift operator}, defined on $\Ltwo$  as the isometric map 
	$$
	\Sh\colon h \in L_{+}^{2}(\T)  \longmapsto \mathrm{e}^{i x} h \in L_{+}^{2}(\T)\,. 
	$$
	Its adjoint in $\Ltwo$ is given by $$\displaystyle \Sh^{*}:h \in L_{+}^{2}(\T)  \longmapsto S^{*} h= T_{\mathrm{e}^{-i x}}h= \Pi(\mathrm{e}^{-i x} h )\in L_{+}^{2}(\T)\,\ .$$
	In particular, we have 
	\begin{equation}\label{SS*}
		\Sh^{*} \Sh=\Id, \qquad \Sh \Sh^{*}=\Id-\langle\,\cdot \mid 1\rangle 1 ,
	\end{equation}
	leading to the fact that the shift map $S$ is injective but not surjective.
	Pointing out that the Hardy space can be defined with different approaches, for instance,
	$$
	\mathbb H_2(\mathbb D):=\lracc{u\in \mathrm{Hol}(\D)\,;\,\sup_{0\leq r<1}\int_0^{2\pi} \va{u(r\eee^{i\theta})}^2\,\dfrac{d\theta}{2\pi}<\infty}\,,
	$$ 
	which is equivalent 
	% $$
	%     L^2_+(\T):=\lracc{u\in L^2(\T)\,;\ \fr{u}(n)=0\,,\,\qlq n\in \Z_{\leq -1}}\equiv\, \Pi(L^2(\T))\,,
	% $$   
	via the isometric isomorphism
	$$
	u(z)=\sum_{k\geq 0}\fr{u}(k)z^k \longmapsto u^*(x):=\sum_{k\geq 0}\fr{u}(k)\eee^{ikx}\,,
	$$
	to the Hardy space $L^2_+(\T)$ defined in \eqref{L2+}\,,
	then one could read the shift operator acting as multiplication by $z\,.$
	In what follows, we use indifferently $u$ and the boundary function $u^*$\,, by making a slight abuse of notation and denoting both by~$u$\,.

	\begin{center}
		***
	\end{center}
	\vskip0.2cm
	Coming back to the problem, we need some commutator identities to obtain the explicit formula. This is the purpose of the next Lemma.

	\begin{lemma}\label{commutator} 
		Let $u\in H^s_+(\T)$\,, $s>\frac32\,,$  then
		\begin{gather}\label{commutateurLuS}
			[S^*,L_u]  = S^*-\ps{\,\cdot}{u}S^*u\,,
			\\
			[S^*, B_u]=i\Big( S^*L_u^2 \,-\, (L_u+\Id)^2S^* \Big)\,.\notag
 		\end{gather}
	\end{lemma}
	
	\begin{proof}
		The first identity is a direct consequence of proving 
		\be\label{[L_u,S]}
		L_uS =S L_u+S-\ps{\,\cdot}{S^*u}u
		\,,
		\ee
		and taking the adjoint of all these operators in $L^2_+(\T)\,$.
		Recall $L_u=D-T_uT_{\bar u}$\,. On the one hand, we have by the Leibniz rule  $D(S h)=S(\Id+ D)h$\,, for all $h \in H_{+}^{1}(\T)$\,. On the other hand, observe for all $f\in L^2(\T)$,  
		$$
		\Pi\left(\Sh f\right)=S \Pi(f)+\left\langle S f \mid 1\right\rangle\,.
		$$
		In particular, for $f= h\bar u$\,, we infer
		\be\label{commutateurTuS}
		T_{\bar{u}}(\Sh h)=\Sh T_{\bar{u}} h+\langle  \Sh h \mid u\rangle.
		\ee
		Hence, taking into consideration that the operators $S$ and $T_u$ commute, we deduce identity \eqref{[L_u,S]}\,.
		
		\vskip0.1cm
		Now, to prove the second point of \eqref{commutateurLuS} we use the first point. Recall  that $B_u=T_{u}T_{\partial_x\bar{u}}-T_{\partial_x{u}}T_{\bar{u}} +i(T_{u}T_{\bar{u}})^2$\,, and %notice 
		by \eqref{commutateurTuS} we have $[T_{\bar u}, S]=\ps{\,\cdot}{S^*u}$, in other words, $[S^*, T_u]=\ps{\,\cdot}{1}S^*u $. 
		Thus, after noting that $S^*$ and $T_{\bar u}$ commute, we deduce
		\begin{gather*}
			[S^*,T_uT_{\p_x\bar u}]=\ps{\,\cdot}{\p_x u}S^*u\,.
			\\
			[S^*,T_{\p_x u}T_{\bar u}]=\ps{\,\cdot}{ u}S^*\p_x u\,,
			\\
			[S^*,(T_{u}T_{\bar{u}})^2]=\ps{\,\cdot}{ T_{u}T_{\bar{u}}u}S^*u+T_{u}T_{\bar{u}}(\ps{\,\cdot}{u} S^*u)\,.
		\end{gather*}
		As a result, 
		$$
		[S^*, B_u]=\ps{\,\cdot}{\p_x u}S^*u-\ps{\,\cdot}{ u}S^*\p_x u
		+i\ps{\,\cdot}{ T_{u}T_{\bar{u}}u}S^*u
		+iT_{u}T_{\bar{u}}\,(\ps{\,\cdot}{u} S^*u)\,.
		$$
		Using the adjoint Leibniz rule $S^*D=(D+\Id)S^*$ and since $L_u=D-T_{u}T_{\bar{u}}$\,, we infer
		\begin{align*}
			[S^*, B_u]
			=&\,
			-i\ps{\,\cdot}{L_uu} S^*u- iL_u (\ps{\,\cdot}{u} S^*u)-i\ps{\,\cdot}{u} S^*u  
			\\
			=&\,
			-i\,(\ps{\,\cdot}{u} S^*u) L_u\,-\,i(L_u+\Id)(\ps{\,\cdot}{u} S^*u)\,.
		\end{align*}
		We conclude by the first identity of \eqref{commutateurLuS} that $-\ps{\,\cdot}{u} S^*u=S^*L_u-L_uS^*-S^*$ and hence
		$$
		[S^*, B_u]=i\Big( S^*L_u^2 \,-\, (L_u+\Id)^2S^* \Big)\,.
		$$
	\end{proof}

	\begin{prop}\label{Lu self-adjoint}
		Let $u(t)\in H^s_+(\T)\,$, $s>\frac32\,$. The Lax operator $(L_{u(t)},H^1_+(\T))$ is a self--adjoint operator with a discrete spectrum bounded from below. Moreover, $B_{u(t)}$ is a skew--symmetric bounded operator on $\Ltwo\,$.
	\end{prop}
	
	\begin{proof}The proof is a direct consequence of Kato--Rellich's theorem.
		Indeed, the differential operator $(D,H^1_+(\T))$ is a positive self--adjoint operator on the Hardy space $L^2_+(\T)$. In addition,  $T_u T_{\overline{u}}$ is relatively bounded with respect to $D$, since for all $h \in H^1_+(\T)$\,, $$
		\left\| T_{u}T_{\bar{u}}h\right\|_{L^{2}}
		\leq\left\| u\right\|^2 _{L^{\infty}} \left\| h\right\| _{L^{2 }}\ \leq \eps\Ldeux{Dh}+\left\| u\right\|^2 _{L^{\infty}}\Ldeux{h},\quad 0\leq\eps<1.$$
		% Then, according to Kato--Rellich's theorem,  $L_u$ is a self-adjoint operator with a bounded spectrum from below.
		Furthermore, the spectrum of $L_u$ is discrete since the resolvent of $L_u$ is compact by the Rellich--Kondrachov theorem. And it is bounded from below as the operator $L_u$ is a semi--bounded operator.
		Besides, one can easily observe by definition of $B_u=T_uT_{\partial_x\overline u}-T_{\partial_xu}T_{\overline u} +i(T_uT_{\overline u})^2\,,$ that this operator is a skew--symmetric operator.
	\end{proof}

	% \vskip0.25cm
	In view of the previous proposition, 
	we denote by $(\la_n(u))_{n\geq 0}$ the eigenvalues of $L_u$ ordered by increasing modulus, and taking into account their multiplicity
	$$
	\la_0(u)\leq \la_1(u) \leq \la_2(u)\leq\ldots \leq \la_n(u)\leq\ldots
	$$

	\begin{Rq}[Isospectral property]\label{val propres conservées}
		As discovered in the modern theory of integrable systems \cite{GGKM67} and reformulated by \cite{Lax68}, the eigenvalues of a Lax operator are integrals of motion of the associated equation.
  In fact, any Lax operator satisfies
the isospectral property, namely, there exists a one-parameter family of unitary
operators $U(t)$ 
  such that $U(t)^{-1}L_{u(t,\cdot)}U(t)$  is independent of $t$\,. That is,
		\be\label{Isospectral property}
		U(t)^{-1}L_{u(t)}U(t)=L_{u_0}\,.
		\ee
  % satisfies the \textit{isospectral property}, namely, the spectrum of $L_u$ is conserved along the flow. Indeed, there exists 
		% In addition, this family of unitary operators $U(t)$ is solution of the Cauchy problem 
		This implies, that the eigenvalues $(\la_n(u))$ of $L_u$ are all conserved along the flow of \eqref{CS+}\,. Or in other words, for all $n\in\Nzero$\,, $\la_n(u(t))=\la_n(u_0)$ for all $t\in \R\,.$
	\end{Rq}

	% \vskip0.25cm
	The following lemma provides a rewrite of the Calogero--Sutherland DNLS equation focusing on \eqref{CS+} in terms of the Lax operators $L_u$ and $B_u$\,. This will certainly be useful during the proof of the dynamical explicit formula.

	\begin{lemma}\label{Reformulation de CS à l'aide des op. de Lax}
		Given $u\in \mathcal C([-T,T],H^s_+(\T))$\,, $s>\frac32\,,$ a solution of \eqref{CS+} equation, then $$
		\p_t u=B_u u -iL^2_uu\,.$$
	\end{lemma}
	
	\begin{proof}
		By definition of $B_u:=T_uT_{\partial_x\overline u}-T_{\partial_xu}T_{\overline u} +i(T_uT_{\overline u})^2\,,$  
		\begin{align*}\p_t u-B_uu=&\, i\p_x^2 u+2iD_+(|u|^2)u-T_uT_{\p_x\overline{u}}u+T_{\partial_xu}T_{\overline u}u -i(T_uT_{\overline u})^2u
			\\
			=&-i\left[D^2 u-2u\cdot D\Pi(|u|^2)+u\cdot \Pi(D\overline{u}\cdot u)-Du\cdot \Pi (|u|^2)+(T_uT_{\overline u})^2u\right]\,.
		\end{align*}
		Applying Leibniz's rule on $D\left(u\cdot\Pi(|u|^2)\right)$\,, we infer
		\begin{align*}
			\p_t u-B_uu
			=&-i\left[D^2 u-D\left[\Pi(|u|^2)\cdot u\right]+u\Pi(D\overline{u}\cdot u)-uD\Pi (|u|^2)+(T_uT_{\overline u})^2u\right]\,.
		\end{align*}
		Again, using Leibniz's rule on the term $D\Pi(|u|^2)$\,, $$
		\p_t u-B_uu
		=-i\left[D^2 u-D(T_uT_{\overline{u}}u)-u\Pi(\overline{u}\cdot Du)+(T_uT_{\overline u})^2u\right]
		=-i L^2_uu\ .$$
	\end{proof}

	Following \cite{Ger22} and \cite{GG15}, we derive the explicit formula for the solution of the Calogero--Sutherland DNLS focusing equation.

	\begin{prop}[The explicit formula]\label{The inverse dynamical formula}
		Given $u_0\in H^s_+(\T)$\,, $s>\frac32\,,$ the solution of the focusing Calogero--Sutherland DNLS equation \eqref{CS+} is given by
		\[
		u(t,z)=\ps{(\Id-z\eee^{-it}\eee^{-2itL_{u_0}}S^*)^{-1}\,u_0}{1}\,,\qquad \qlq \,z\in \D\,.
		\]
	\end{prop}
	
	\begin{proof}
		Since  $u(t,\cdot)\in H^s_+(\T)$\,, $s>\frac32\,,$ for all $t\in [-T,T]$\,, then   for all $z\in\D$
		$$
		u(t,z)=\sum_{k=0}^{\infty} \widehat{u}(t,k) z^{k}=
		\sum_{k=0}^{\infty}\left\langle u(t)\mid \Sh^{k} 1\right\rangle z^{k}=\sum_{k=0}^{\infty}\left\langle\left(S^{*}\right)^{k} u(t) \mid 1\right\rangle z^{k}\,,
		$$
		where by the Neumann series  of
		$$
		\sum_{k=0}^{\infty}\left(z \Sh^{*}\right)^{k}=\left(\mathrm{Id}-z \Sh^{*}\right)^{-1}\,,
		$$ 
		we infer
		\be\label{formule inversion spectral Hardy}
		u(t,z)=\left\langle\left(\Id-z \Sh^{*}\right)^{-1} u(t) \mid 1\right\rangle, \qquad \forall\, z\in\D \,.
		\ee
        Now, consider  a one--parameter family $U(t)$ solution of the Cauchy problem
        \be\label{Cauchy problem unitary operator}
		\begin{cases}\frac{d}{dt}U(t) =B_{u(t,\cdot)}\, U(t)\\
			U( 0) =\Id.
		\end{cases}
		\ee 
  Observe that $U(t)$ is a unitary operator since $B_u$ is skew-adjoint. Moreover, using the Lax pair structure of Proposition~\ref{The lax pair prop}\,, and \eqref{Cauchy problem unitary operator}\,,
  $$
    \frac{d}{dt}(U(t)^* L_{u(t)}U(t))=0\,,
  $$
 and thus 
 \be\label{U(t)*LuU(t)=Lu0}
    U(t)^* L_{u(t)}U(t)=L_{u_0}\,.
 \ee
 Hence, by applying $U(t)^*$ to both sides of the inner product \eqref{formule inversion spectral Hardy}\,,
		\begin{align}\label{Apply U(t) in the inversion formula}
			u(t,z)
			=&\,\left\langle U(t)^*\left(\Id-z \Sh^{*}\right)^{-1} u(t) \mid U(t)^* 1\right\rangle
			\\
			=&\,\left\langle \,\left(\Id-z U(t)^*S^{*}U(t)\right)^{-1} U(t)^* u(t) \mid U(t)^* 1\right\rangle\,.\notag
		\end{align}
  The aim is to express differently $U(t)^*S^{*}U(t)\,,$ $\;U(t)^* u(t)$ and $U(t)^* 1\,.$ Using \eqref{Cauchy problem unitary operator} and since $B_u$ is a skew--adjoint operator (Proposition~\ref{Lu self-adjoint}), we find
		\begin{itemize}
			\item $\frac{d}{dt}[U(t)^*1] 
			\,=\, 
			-U(t)^*B_{u(t)}1
			\,=\,
			-iU(t)^*L^2_{u(t)}1$
			\\
			\item $\frac{d}{dt}[U(t)^*u(t)] 
			= 
			-U(t)^*B_{u(t)}u(t)+U(t)^*\p_tu(t)
			=
			-iU(t)^*L^2_{u(t)}u(t)$ by Lemma~\ref{Reformulation de CS à l'aide des op. de Lax}
			\\
			\item $\frac{d}{dt}[U(t)^*S^*U(t)]  =-U(t)^*B_{u(t)}S^*U(t)+U(t)^*S^*B_{u(t)}U(t)
			=U(t)^*\,[S^*, B_{u(t)}]\,U(t)$\,,
		\end{itemize}
		where the third point is equal to $$
		\frac{d}{dt}[U(t)^*S^*U(t)] =iU(t)^*\Big( S^*L_{u(t)}^2 \,-\, (L_{u(t)}+\Id)^2S^* \Big)U(t)\,,
		$$ by Lemma~\ref{commutator}\,. 
		Therefore, applying the identity $U(t)^*L_{u(t)}=L_{u_0}U(t)^*$ of \eqref{U(t)*LuU(t)=Lu0}\,, we deduce
		\begin{itemize}
			\item $ \frac{d}{dt}[U(t)^*1] 
			= -iL^2_{u_0}[U(t)^*1]$
			\\
			\item $\frac{d}{dt}[U(t)^*u(t)] 
			= -iL^2_{u_0}[U(t)^*u(t)]$
			\\
			\item $\frac{d}{dt}[U(t)^*S^*U(t)] 
			= i\Big([U(t)^*S ^*U(t)]L^2_{u_0}
			\,-\,
			(L_{u_0}+\Id)^2 [U(t)^*S ^*U(t)]
			\Big)\,.$
		\end{itemize}
		As a consequence,
		\be\label{U(t)*1}
		U(t)^*1=\eee^{-itL^2_{u_0}}1\,,
		\qquad 
		\qquad 
		U(t)^*u(t)=\eee^{-itL^2_{u_0}}u_0\,,
		\ee
		and 
		\be\label{U(t)*S*U(t)}
		U(t)^*S^*U(t)=\eee^{-it(L_{u_0}+\Id)^2}S^*\eee^{itL_{u_0}^2}\,.
		\ee
		Combining \eqref{Apply U(t) in the inversion formula}\,, \eqref{U(t)*1} and \eqref{U(t)*S*U(t)}, the claimed formula follows.
		% $$
		% In addition, applying the second identity of Lemma~\ref{commutator},
		% $$
		%     \frac{d}{dt}[U(t)^*S^*U(t)] =iU(t)^*\Big( S^*L_u^2 \,-\, (L_u+\Id)^2S^* \Big)U(t)
		% $$
	\end{proof}
	\vskip0.5cm
	
	\subsection{Global well--posedness of (CS\texorpdfstring{$^+$}{+}) 
		in \texorpdfstring{$H^s_+(\T)\,,$}{Hs+(T),} \texorpdfstring{$s>\frac32$ }{s>3/2}}
	\label{GWP-in-H2-if-u0<1}
	To prove the global well--posedness of \eqref{CS+}\,, we  need to derive some conservation laws and energy estimates.
	
	\begin{lemma}[Conservation laws]\label{loi de conservation}
		Let $u\in \mathcal C \big([-T,T],H^{r}_+(\T)\big)\,,$ $r>\frac32$\,, solution of \eqref{CS+}\,. For all $\la>\!>0\,,$ the family $\lracc{\mathcal{H}_s(u):=\ps{(L_u+\la)^s\,u}{u}\,;\, 0\leq s\leq 2r}$ is conserved by the flow of \eqref{CS}\,. 
	\end{lemma}
	
	\begin{Rq}\label{extension conservation laws}\hfill
 \begin{itemize}
     \item Using complex interpolation method \cite[Chapter I. 4.]{Ta81}, one can observe as demonstrate in Proposition~\ref{controle loi de conservation}\,, that the  $\mathcal{H}_s(u)\leq C\|u\|_{H^\frac{s}{2}}^2\,$.
     \item The condition $r>\frac32$ is to guarantee the existence of $u(t)$\,. It can be omitted once we prove in Section~\ref{Proof of th 3 and 4} that the flow $u_0\in\B{1}\cap H^r_+(\T)\mapsto u(t)\in H^r_+(\T)$ exists for all $r\geq0$.
 \end{itemize}
	\end{Rq}
	\begin{proof}
		Given $u\in \mathcal C \big([-T,T],H^{r}_+(\T)\big)\,,$ $r>\frac32\,,$ solution of \eqref{CS+}\,, we consider the unitary operator $U(t)$ defined in \eqref{Cauchy problem unitary operator}\,. Then, by \eqref{U(t)*1}\,, we know that $ U(t)^*u(t)=\eee^{-itL^2_{u_0}}u_0\,.$
		And, since $L_u$ is a self--adjoint operator by Proposition~\ref{Lu self-adjoint}\,, we infer by \eqref{U(t)*LuU(t)=Lu0}\,,
		$$
		U(t)^*(L_{u(t)}+\la)^sU(t)=(L_{u_0}+\la)^s\,.
		$$     
		Therefore, for all $0\leq s\leq 2r\,,$
		\begin{align*}
			\mathcal H_s(u(t))
			=&\,\ps{U(t)^*(L_{u(t)}+\la)^s{u(t)}}{U(t)^*{u(t)}}
			\\
			=&\,\ps{(L_{u_0}+\la)^sU(t)^*{u(t)}}{U(t)^*u(t)}
			\\
			=&\,\langle (L_{u_0}+\la)^s \eee^{-itL^2_{u_0}}u_0\mid \eee^{-itL^2_{u_0}}u_0\rangle
		\end{align*}
		As a consequence, $\mathcal{H}_s(u(t))=\mathcal{H}_s(u_0)$ as $(L_{u_0}+\la)^s$ and $\eee^{-itL^2_{u_0}}$ commute.
		
	\end{proof}

	\begin{Rq} Using the identity $U(t)^*1=\eee^{-itL^2_{u_0}}1$ of \eqref{U(t)*1} and repeating the same proof of Lemma~\ref{loi de conservation}\,, one can also deduce for $\la>\!>0\,,$ that the quantities $\ps{(L_u+\la)^q1}{u}$ and $\ps{(L_u+\la)^p1}{1}$ are conserved by the flow. Another way to show this, is to observe by definition of $L_u=D-T_uT_{\bar u}$  we have $L_u1=-\ps{1}{u} u$
		and the average $\ps{u}{1}$ is conserved  along the evolution, since
		\[
		\p_t\ps{u}{1}=i\ps{\p_x^2 u}{1}+2i\ps{D\Pi(|u|^2)}{\overline{u}}=0\,.
		\]
	\end{Rq}
	\vskip0.2cm
	To prove the energy estimates and for future requests, we  need the following lemma.
	
	\begin{lemma}\label{inegality Pi(u bar h)}
		Let $h\in H^{\frac12}_{+}(\T)\,$, $u\in\Ltwo\,,$
		\be\label{inegalite-Tu-bar-h}
		\Ldeux{T_{\bar{u}}h}^2\leq \left(\ps{Dh}{h}+\Ldeux{h}^2\right)\Ldeux{u}^2\,.
		\ee
	\end{lemma}
	
	\begin{proof}
		By Parseval's identity,  
		$$ 
		\Ldeux{T_{\bar{u}}h}^2=\sum_{n\geq0}\big\lvert\fr{T_{\bar{u}}h}(n)\big\rvert^2\,,
		$$
		where 
		\be\label{Fourier coefficient of Pi(u bar h)}
		\fr{T_{\bar{u}}h}(n)=\fr{\Pi(h\bar u)}(n)=\sum_{p\geq0}\fr{h}(n+p)\overline{\fr{u}(p)}\,.
		\ee
		Applying Cauchy--Schwarz's inequality, we infer
		$$ 
		\Ldeux{T_{\bar{u}}h}^2
		\leq
		\Ldeux{u}^2\,
		\sum_{p\geq0}\sum_{n\geq0}\va{\fr{h}(n+p)}^2\,.
		$$
		Set $k=n+p$, then
		$$
		\sum_{p\geq0}\sum_{n\geq0}\va{\fr{h}(n+p)}^2=\sum_{k\geq 0}(k+1)\va{\fr{h}(k)}^2= \ps{Dh}{h}+\Ldeux{h}^2\,.
		$$
	\end{proof}
	
	\begin{Rq}\label{Hankel op}
		\hfill
		\begin{enumerate}
			\item Recall that the embedding $H^{\frac12}(\T)\hookrightarrow L^{\infty}(\T)$ fails to be true. However, taking the potential $u$ as an element of the Hardy space $L^2_+(\T)$\,, improved the estimate from $
			\|T_{\bar u}h\|_{L^2}\leq  \|h\|_{L^\infty} \|u\|_{L^2}
			$ to \eqref{inegalite-Tu-bar-h}\,.
			\vskip0.15cm
			\item  \label{Hankel op point 2} From \eqref{Fourier coefficient of Pi(u bar h)}, one could see that, for all $h\in H^\frac12_+(\T)\,,$ the Hilbert--Schmidt norm of the antilinear operator $u\in\Ltwo\mapsto T_{\bar u}h$ is given by
			$$
			\hskip1.5cm \|\Pi(\overline{\,\cdot}\,h)\|_{HS}^2
			=\sum_{p\geq0}
			\sum_{n\geq0}\va{\fr{h}(n+p)}^2=\sum_{k\geq 0}(k+1)\va{\fr{h}(k)}^2= \ps{Dh}{h}+\Ldeux{h}^2\,.
			$$ 
			In particular, we have $u\mapsto T_{\bar u}h$ is a compact antilinear operator in $\Ltwo\,.$ 
			\vskip0.1cm
			\item The inequality~\eqref{inegalite-Tu-bar-h} of Lemma~\ref{inegality Pi(u bar h)} is a sharp inequality since its proof relies on a simple application of the Cauchy--Schwarz inequality.
			In particular, when $h=u$\,, inequality~\eqref{inegalite-Tu-bar-h} is an equality, if and only if 
			$$
			\qquad u(z)=\frac{c}{1-qz}\,, \qquad \va{q}<1\,, \quad c\in\C
			$$
   --which corresponds to the profile of a traveling wave of \eqref{CS+} when $c=\sqrt{1-\va{q}^2}\,$ \cite{Ba23}\,.
			Indeed, following arguments used in \cite[Lemma~1]{GG08}, one observe that  the Cauchy--Schwarz inequality applied to \eqref{Fourier coefficient of Pi(u bar h)}\ is an equality, if and only if  for all $n\geq 0\,,$ there exists $c_n\in\C$ such that
			\be\label{fr u(n+p) fr u(p)}
			\qquad \fr{u}(n+p)=c_n\,\fr{u}(p)\,, \qquad \qlq p\geq 0\,.
			\ee
			Hence, if $\fr{u}(1)\neq 0$ and $\fr{u}(0)\neq0$\,,
			$$
			\qquad c_n=\frac{\fr{u}(n+1)}{\fr{u}(1)}=\frac{\fr{u}(n)}{\fr{u}(0)}\;,
			$$
			leading to,
			$$
			\qquad \fr{u}(n)=\left(\frac{\fr{u}(1)}{\fr{u}(0)}\right)^n\,\fr{u}(0)\,,\qquad \qlq n\in\N\,.
			$$
			Therefore, the sequence $(\widehat{u}(n))$ is  a geometric progression with common ratio $q:=\frac{\fr{u}(1)}{\fr{u}(0)}\,,$ where $0<\va{q}<1$ since 
			$\sum_{n=0}^\infty \lvert\widehat{u}(n)\rvert^2<+\infty$\,.
			Hence,
			$$
			\qquad u(z)=\sum_{n=0}^\infty \widehat{u}(n)z^n=\frac{\widehat{u}(0)}{1-qz}\,.
			$$
			Now, if $\fr{u}(1)=0$ or $\fr{u}(0)=0$ then by \eqref{fr u(n+p) fr u(p)} we infer for $p=0$ or $1$\,, $\, u=\fr{u}(0)\,.$
			% Finally, note that when $\fr{u}(0)=\sqrt{1-\va{q}^2}$ i.e. $u(z)=\frac{\sqrt{1-\va{q}^2}}{1-qz}$ then $u$ is a profile for a traveling wave of \eqref{CS+}\,.
		\end{enumerate}
		
	\end{Rq}
	
	We recall that $\mathcal{B}_{L^2_+}(r)$ denotes the open ball of $\Ltwo$ centered at the origin, with radius $r>0$\,. And we denote for any $s\in\R$ by $\lfloor s\rfloor:=\max\{k\in\Z\,;\, k< s\}$\,.
	
	\begin{prop}\label{controle loi de conservation}
		Let $u_0\in \B{1}\cap H^r_+(\T)\,,$ $r>\frac32\,.$ Then, for all $\la>\!>0$\,,  there exists  $\,C=C(\|u_0\|_{H^{\frac{\lfloor 2r\rfloor}{2}}}\,,\lambda)>0$ independent of $t$\,, such that for every 
  $f\in H^s_+(\T)\,,$ 
  % $\,,$ 
		\be\label{inegalite pour tout s}
		\frac{1}{C}\,\|f\|_{H^s}\leq \|(L_{u(t)}+\la)^sf\|_{L^2}\leq C\|f\|_{H^s}\,, \qquad \qlq\, 0\leq s\leq \frac{\lfloor 2r\rfloor +1}{2}\,.
		\ee
	\end{prop}
	
	\begin{Rq}\label{r>3/2}
		The condition $r>\frac32$ is to guarantee the existence of $u(t)$\,. It can be replaced by $r\geq 0$ once we prove in Section~\ref{Proof of th 3 and 4} that the flow $u_0\in\B{1}\cap H^r_+(\T)\mapsto u(t)\in H^r_+(\T)$ exists for all $r\geq0$.
	\end{Rq}
	
	\begin{proof}
		The proof is done by induction on every interval of length $1/2$\,.
		\vskip0.1cm
		\noindent
		\underline{\textbf{Step 1 :} $\,s\in[0,\frac12]\,.$} Let $s=\frac12$\,, we have by definition of $L_u=D-T_uT_{\bar u}\,,$
		\begin{align}\label{control H12 de la solution}
			% \|(L_{u(t)}+\la)^\frac12 f\|_{L^2}^2
			 \ps{(L_u+\la)f}{f}
			=&\,\ps{Df}{f}-\|T_{\bar u} f\|_{L^2}^2 +\la\|f\|_{L^2}^2\;.
		\end{align}
		Applying the sharp inequality of Lemma~\ref{inegality Pi(u bar h)}\,,
		\begin{align}\label{inequality leading to assumption on the initail data}
			\|(L_{u}+\la)^\frac12 f\|_{L^2}^2
			\geq (1-\|u\|_{L^2}^2)\ps{Df}{f}
			+(\la-\|u\|_{L^2}^2)\|f\|_{L^2}^2\,.
		\end{align}
		Thus, we infer since $\|u\|_{L^2}=\|u_0\|_{L^2}<1\,,$
		\[
		\|(L_{u(t)}+\la)^\frac12 f\|_{L^2}^2\geq \frac{1}{C^2}\|f\|_{H^\frac12}^2\,,
		\]
		where $C=C(\|u_0\|_{L^2}\,, \, \lambda)>0$ is a positive constant independent of $t\,.$ On the other hand, using the definition of $L_u$\,, it is easy to see that 
		$$
		\|(L_{u(t)}+\la)^\frac12f\|_{L^2}\leq C\|f\|_{H^\frac12}\,.
		$$
		Therefore, by complex interpolation \cite[Chapter I. 4]{Ta81}\,, we deduce that inequality~\eqref{inegalite pour tout s} holds true for all $s\in[0,\frac12]\,.$
		\vskip0.1cm
		\noindent
		\underline{\textbf{Step 2 :} Uniform bounds on $\|u(t)\|_{L^p}\,,$ $p\in[2,\infty)$\,. } 
		Applying step 1, and using the conservation laws of Lemma~\ref{loi de conservation}\,, we infer for $f=u\,,$
		\begin{align*}
			\frac{1}{C}\,\|u(t)\|_{H^s}
			\leq \|(L_{u(t)}+\la)^su(t)\|_{L^2}
			=\|(L_{u_0}+\la)^su_0\|_{L^2}\leq C\|u_0\|_{H^s}\,,
		\end{align*}
		for all $s\in[0,\frac12]\,.$ Therefore,
		$
		\sup_{t\in\R}\|u(t)\|_{H^\frac12}\lesssim \|u_0\|_{H^\frac12}\,,
		$
		and thus by Sobolev embedding
		\[
		\sup_{t\in\R}\|u(t)\|_{L^p}\lesssim \|u_0\|_{H^\frac12}\,, \qquad \qlq p\in[2,\infty)\,.
		\]
		% for all $p\in[2,\infty)\,.$
		\vskip0.1cm
		\noindent
		\underline{\textbf{Step 3 :}  $\,s\in[\frac12,1]\,.$ } As in Step~1, the idea is to prove that it is true for $s=1$ and then by complex interpolation, infer that it is true for all $s\in[\frac12,1]\,.$ Let $s=1$\,, by \eqref{inequality leading to assumption on the initail data} of Step 1,
		\begin{align*}
			\Ldeux{(L_{u(t)}+\la)f}^2
			=&\,\ps{(L_{u(t)}+\lambda)f}{L_{u(t)}f}+\la\ps{(L_{u(t)}+\la)f}{f}
			\\
			\geq&\,
			\|L_{u(t)}f\|_{L^2}^2
			+\la\Big( (1-\|u\|_{L^2}^2)\|f\|_{\dot{H}^\frac12}^2
			+(\la-\|u\|_{L^2}^2)\|f\|_{L^2}^2\Big)
			\\
			&\hskip2cm
			+\la\Big((1-\|u\|_{L^2}^2)\|f\|_{\dot{H}^\frac12}^2
			-\|u\|_{L^2}^2\|f\|_{L^2}^2\Big)
			\\
			=&\,\|Df\|_{L^2}^2+\|uT_{\bar u}f\|_{L^2}^2-2 \mrm{Re}\ps{Df}{uT_{\bar u}f}
			\\
			&\,\hskip1.45cm
			+2\la(1-\|u\|_{L^2}^2)\|f\|_{\dot{H}^\frac12}^2
			+\la(\la-2\|u\|_{L^2}^2)\|f\|_{L^2}^2\,.
		\end{align*}
		Using Young's inequality, we deduce
		\begin{align}\label{inequality of step 3}
			\Ldeux{(L_{u(t)}+\la)f}^2
			\geq&\,
			(1-\eps)\|Df\|_{L^2}^2+(1-C_\eps)\|uT_{\bar u}f\|_{L^2}^2
			\\
			&\,
			+2\la(1-\|u\|_{L^2}^2)\|f\|_{\dot{H}^\frac12}^2
			+\la(\la-2\|u\|_{L^2}^2)\|f\|_{L^2}^2\,.\notag
		\end{align}
		Now, applying Cauchy--Schwarz's inequality on $\|uT_{\bar u}f\|_{L^2}$ and since $\|u\|_{L^8}$ and $\|u\|_{L^4}$ are uniformly bounded by Step 2., we infer 
		\[
		\Ldeux{(L_{u(t)}+\la)f}^2\geq \frac{1}{C^2}\,\|f\|_{H^1}^2\;, \qquad C=C(\|u_0\|_{H^\frac12}, \lambda)\,.   
		\]
		On the other hand, we have by definition of $L_u$\,,
		\[
		\Ldeux{(L_{u(t)}+\la)f}\leq C\|f\|_{H^1}\,. 
		\]
  Therefore, inequality \eqref{inegalite pour tout s} holds for $s=1$ and thus by complex interpolation \cite{Ta81}, it holds for all $s\in[\frac12,1]\,.$
  \vskip0.1cm
  \noindent
  \underline{\textbf{Step 4 :} Uniform bounds on $\|u(t)\|_{H^1}$ }  Since \eqref{inegalite pour tout s} holds for $s=1$ then we infer by repeating the same proof of Step~2 that 
  \[
    \sup_{t\in\R}\|u(t)\|_{H^1}<\infty\,.
  \]
   \vskip0.1cm
  \noindent
  \underline{\textbf{Step 5 :} $s\in[1,\frac32]$ } Again, we prove that inequality~\eqref{inegalite pour tout s} is true for $s=\frac32\,,$ then by complex interpolation, we deduce it for all $s\in[1,\frac32]\,.$
Let $s=\frac32\,.$ By applying inequality \eqref{inequality of step 3} of Step~3\,,
\begin{align*}
    \ps{(L_{u(t)}+\lambda)^3f}{f}
		=&\,\ps{(L_{u(t)}+\lambda)^2f}{L_{u(t)}f}\,+\,\la\ps{(L_{u(t)}+\la)^2f}{f}
			\\
        \geq &\,\ps{(D-uT_{\bar u}+\lambda)^2f}{(D-uT_{\bar u})f}+ \lambda\Big[
        (1-\eps)\|Df\|_{L^2}^2
			\\
			&\,+(1-C_\eps)\|uT_{\bar u}f\|_{L^2}^2
			+2\la(1-\|u\|_{L^2}^2)\|f\|_{\dot{H}^\frac12}^2
			+\la(\la-2\|u\|_{L^2}^2)\|f\|_{L^2}^2\Big]
\end{align*}
Now, expanding the first inner product on the right--hand side, and using the Cauchy--Schwarz inequality, Young's inequality, and Sobolev embedding\,, we obtain
\begin{align*}
    \ps{(L_{u(t)}+\lambda)^3f}{f}\gtrsim&\, \ps{D^3f}{f}\,+p_1( \|u\|_{H^1}^2, \lambda)\ps{D^2f}{f}
    \\
    &\,+p_2( \|u\|_{H^1}^2, \lambda)\ps{Df}{f}
    \,+\,p_3(\|u\|_{H^1}^2, \lambda)\|f\|_{L^2}^2
\end{align*}
where for all $j=1,2,3,$ $\,p_j(\|u\|_{H^1}^2, \lambda)$ is a positive polynomial for $\lambda>\!>0\,.$ Therefore,
there exists $C=C(\|u_0\|_{H^1}, \lambda)>0$ such that 
\[
		\Ldeux{(L_{u(t)}+\la)^{\frac32}f}\geq \,\frac{1}{C}\,\|f\|_{H^\frac32}\,.
		\]
Finally, by repeating the same previous procedure, 
  we infer that  inequality \eqref{inegalite pour tout s} holds true for all $s\geq 0$\,.
		% We conclude by complex interpolation as in Step 1. that inequality~\eqref{inegalite pour tout s} holds true for all $s\in[\frac12,1]\,.$ Then,  by induction on  all the intervals $s\in[\frac{n}{2},\frac{n}{2}+\frac12]$\,, $n\in\N$\,, we have that inequality \eqref{inegalite pour tout s} is true for all $s\geq 0\,$. 
	\end{proof} 
	
	\begin{theorem*}\ref{GWP k geq 4}\textbf{.}
		For all $r>\frac32\,,$ the Calogero--Sutherland DNLS focusing equation \eqref{CS+} is globally well--posed in 
		$H^r_+(\T)\cap \mathcal{B}_{L^2_+}(1)$\,.
		% initial data $u_0\in H^2_+(\T)$\,, .
		Moreover,  the following a--priori bound holds,
		\[
		\sup_{t\in\R}\|u(t)\|_{H^r}\;\leq C\;,
		\]
		where $C=C(\|u_0\|_{H^r})>0$ is a positive constant.
	\end{theorem*}
	
	\begin{proof}
		Let $u_0\in H^r_+(\T)\,.$    Recall by \eqref{flot}\,, there exists a unique solution  $u\in  \mathcal C([-T,T],H^r_+(\T))$\,, $s>\frac32\,,$  satisfying $u(0,\cdot)=u_0\,.$ In addition, in view of the previous proposition, we infer for $\|u_0\|_{L^2}<1\,,$
		\begin{align*}
			\frac{1}{C}\,\|u(t)\|_{H^r}
			\leq \|(L_{u(t)}+\la)^ru(t)\|_{L^2}
			=\|(L_{u_0}+\la)^ru_0\|_{L^2}\leq C\|u_0\|_{H^r}\,.
		\end{align*}
	\end{proof}

	%%%%%%%%%%%%%%%%%%%%%%%%%%%%%%%%%%%%%%%%%%%%%%%%%
	%%%%%%%%%%%%%%%% Extension to L2  %%%%%%%%%%%%%%%
	%%%%%%%%%%%%%%%%%%%%%%%%%%%%%%%%%%%%%%%%%%%%%%%%%
	\vskip1cm
	\section{Extension of the flow of (CS\texorpdfstring{$^+$}{+}) to \texorpdfstring{$\Ltwo$}{L2+(T)} }
	\label{section: Extension du flot}
	\vskip0.2cm
	
	In this section, we establish our main result, which states that, for $\|u_0\|_{L^2}<1\,,$ the flow of \eqref{CS+}
	%initially 
	\be
	\begin{array}{cccc} 
		\mathcal S^+(t)&:H^2_+(\T)&\longrightarrow& H^2_+(\T)\\
		&u_0&\longmapsto&u(t)
	\end{array}\,,
	\ee 
	defined globally on $H^2_+(\T)$ via Theorem~\ref{GWP k geq 4}\,,
	%  for all $t\in\R$\,, 
	can be extended continuously to the critical regularity $L^2_+(\T)\,.$ 
 For this purpose, consider any initial data $u_0\in\Ltwo\,,$ $\|u_0\|_{L^2}<1\,.$ Then we approximate $u_0$  by a sequence $(u_0^\eps)\subseteq H^2_+(\T)\,$.
Thus, in view of Theorem~\ref{GWP k geq 4}\,, the time evolution $u^\eps(t):=\mathcal S^+(t)u_0^\eps$ of  $u_0^\eps$ is well-defined  for all $t\in\R$\,.
	% Our strategy is the following.
	% Starting from $u_0\in L^2_+(\T)\,$, $\|u_0\|_{L^2}<1$\,, we approximate $u_0$  by a sequence $(u_0^\eps)\subseteq H^2_+(\T)\,$.
 % In view of Theorem~\ref{GWP k geq 4}
	% After that, 
	% %  from $(u_0^\eps)$ 
	% we consider the time evolution of  $(u_0^\eps)$\,, i.e. the potentials $(u^\eps(t))$ defined as  $u^\eps(t):=\mathcal S^+(t)u_0^\eps$ for all $t\in\R$\,.
	Our goal is to prove that the sequence $(u^\eps)$ converge to a unique limit $u$ in $\mathcal C(\R,L^2_+(\T))$\,. 
 % independent of the approximate sequence $(u_0^\eps)$\,.  
 This limit potential shall be called ``solution'' to the Cauchy problem \eqref{CS+}\,.
	It  will be uniquely well--defined, 
 regardless of the chosen approximate sequence $(u_0^\eps)$ that approximate $u_0\in\Ltwo\,.$ 
 % Moreover, the solution $u$ 
 % It will also be continuous in time, inducing a global continuous flow 
	% well--defined 
	% on $L^2_+(\T)\,.$ 
 Moreover, 
 it will  satisfies the conservation of the $L^2$--norm (i.e. $\|u(t)\|_{L^2}=\|u_0\|_{L^2}\,$ for all $t\in\R\,$).

	\vspace{0.4cm}

	\begin{center}
		\tikzset{every picture/.style={line width=0.75pt}} %set default line width to 0.75pt        
		
		\begin{tikzpicture}[x=0.7pt,y=0.7pt,yscale=-0.9,xscale=0.9]
			%uncomment if require: \path (0,150); %set diagram left start at 0, and has height of 150
			
			%Straight Lines gauche pointé vers le bas
			\draw    (160,53) -- (160,98) ;
			\draw [shift={(160,100)}, rotate = 270] [color={rgb, 255:red, 0; green, 0; blue, 0 }  ][line width=0.75]    (10.93,-3.29) .. controls (6.95,-1.4) and (3.31,-0.3) .. (0,0) .. controls (3.31,0.3) and (6.95,1.4) .. (10.93,3.29)   ;
			%Straight Lines bas pointé vers la dte 
			\draw    (176,115) -- (294,115) ;
			\draw [shift={(294,115)}, rotate = 180] [color={rgb, 255:red, 0; green, 0; blue, 0 }  ][line width=0.75]    (10.93,-3.29) .. controls (6.95,-1.4) and (3.31,-0.3) .. (0,0) .. controls (3.31,0.3) and (6.95,1.4) .. (10.93,3.29)   ;
			%Straight Lines droite pointé vers le haut 
			\draw    (317,53) -- (317,98) ;
			\draw [shift={(317,52)}, rotate = 90] [color={rgb, 255:red, 0; green, 0; blue, 0 }  ][line width=0.75]    (10.93,-3.29) .. controls (6.95,-1.4) and (3.31,-0.3) .. (0,0) .. controls (3.31,0.3) and (6.95,1.4) .. (10.93,3.29)   ;
			%Straight Lines rouge
			\draw [color={rgb, 255:red, 252; green, 8; blue, 8 }  ,draw opacity=1 ] [dash pattern={on 4.5pt off 4.5pt}]  (177,38) -- (294,38) ;
			\draw [shift={(294,38)}, rotate = 180] [color={rgb, 255:red, 252; green, 8; blue, 8 }  ,draw opacity=1 ][line width=0.75]    (10.93,-3.29) .. controls (6.95,-1.4) and (3.31,-0.3) .. (0,0) .. controls (3.31,0.3) and (6.95,1.4) .. (10.93,3.29)   ;
			%Straight Lines [id:da7547834350402973] 
			\draw [color={rgb, 255:red, 155; green, 155; blue, 155 }  ,draw opacity=1 ] [dash pattern={on 0.84pt off 2.51pt}]  (0,52) -- (490,52) ;
			%Straight Lines [id:da5492720027322395] 
			\draw [color={rgb, 255:red, 155; green, 155; blue, 155 }  ,draw opacity=1 ] [dash pattern={on 0.84pt off 2.51pt}]  (0,24) -- (490,24) ;
			%Straight Lines [id:da7945628690092166] 
			\draw [color={rgb, 255:red, 155; green, 155; blue, 155 }  ,draw opacity=1 ] [dash pattern={on 0.84pt off 2.51pt}]  (0,101) -- (490,101) ;
			%Straight Lines [id:da9630925728964277] 
			\draw [color={rgb, 255:red, 155; green, 155; blue, 155 }  ,draw opacity=1 ] [dash pattern={on 0.84pt off 2.51pt}]  (0,129) -- (490,129) ;
			
			% Text Node
			\draw (152,33) node [anchor=north west][inner sep=0.75pt]    {$u_{0}$};
			% Text Node
			\draw (152,106) node [anchor=north west][inner sep=0.75pt]    {$u_{0}^{\varepsilon }$};
			% Text Node
			\draw (303,106) node [anchor=north west][inner sep=0.75pt]    {$u^{\varepsilon }( t)$};
			% Text Node
			\draw (303,29) node [anchor=north west][inner sep=0.75pt]    {$u( t)$};
			% Text Node
			\draw (232,15) node [anchor=north west][inner sep=0.75pt]   [align=left] {\textcolor[rgb]{0.98,0.04,0.04}{?}};
			% Text Node
			\draw (2,27) node [anchor=north west][inner sep=0.75pt]    {$L_{+}^{2}( \T) \ :$};
			% Text Node
			\draw (2,104) node [anchor=north west][inner sep=0.75pt]    {$H_{+}^{2}( \T) \ :$};

		\end{tikzpicture}
	\end{center}

	\vspace{0.6cm}
	
	\begin{Rq} \label{produit distrib}
	       Note that, due to the presence of the nonlinear term $D\Pi_+(\va{u}^2)u$ in the equation, it may seem intriguing to say that there exists a solution with $L^2$--regularity. Nevertheless, the equation is still well--defined in the distribution sense since the product of two functions with nonnegative frequencies is well--defined and continuous. 
 Indeed, let $\mathscr{D}'_+(\T)$ denotes the following  distribution space
 \[
    \mathscr{D}'_+(\T)=\{u=\sum_{k\geq 0}\fr{u}(k)\eee^{ikx}\;;\,\e M\in\N, \,\va{\fr{u}(k)}\lesssim(1+\va{k}^{2M})^{\frac12}\}\,,
 \]
 and consider two sequences of smooth functions in the Hardy space $f_n\,,\,g_n\in\mathcal{C}^\infty_+(\T)$ such that we suppose
 \[
    f_n\longrightarrow f\quad \text{and}\quad g_n\longrightarrow g \quad\text{in}\ \mathscr{D}'_+(\T)\,.
 \]
 Then, $f_ng_n\longrightarrow fg$ in $\mathscr{D}'_+(\T)$ and for all $k\in\Nzero\,,$ 
 $$
    \fr{fg}(k)=\sum_{\ell=0}^k\fr{f}(k-\ell)\fr{g}(\ell)\,,
$$
 as $\fr{f_n\,g_n}(k)=\sum_{\ell=0}^k\fr{f_n}(k-\ell)\fr{g_n}(\ell)$\,.
 % Besides, one can find multiple examples of nonlinear PDE admitting solutions in $L^2$ and where the nonlinearity does not seem, at first sight, to be well-defined in $\Ltwo\,$. Take, for instance, the family of PDEs $$
	% i\partial_t u=\lambda u +|u|^qu,\quad \lambda\in\mathbb{R}\,,\, q>1\,,
	% $$
 % where for any initial value $u_0\in L^2_+\,,$  
	% $$
	% u(t)=u_0\,\cdot\mathrm{e}^{-i(|u_0|^q+\lambda)\,t}\in L^2\,.
	% $$ 
	% % \vskip0.25cm
	% % In our approach, we use the approximation method described above,
	% % and the limit potential $u(t)$ obtained by this approximation method shall be called ``solution'' to the Cauchy problem \eqref{CS+}\,.
	% % This solution will be uniquely characterized, continuous in time, inducing a global continuous flow 
	% % % well--defined 
	% % on $L^2_+(\T)\,.$ In addition, it  satisfies the $L^2$ invariant mass (i.e. $\|u(t)\|_{L^2}=\|u_0\|_{L^2}\,$ for all $t\in\R\,$).
	\end{Rq}
	~\\
	% \vskip0.3cm
	\subsection{Uniqueness of the limit and weak convergence in \texorpdfstring{$\Ltwo$}{L2+(T)}}
	~\\
	By passing to the limit as $\eps\to0\,,$ it is necessary to first prove that the limit potential $u(t)$ is uniquely well--characterized for all $t\in\R\,,$ and is independent of the choice of the sequence $(u_0^\eps)\subseteq H^2_+(\T)$ that approximate $u_0\in L^2_+(\T)\,.$
	The key point is to use the explicit formula of the solution of the focusing Calogero--Sutherland DNLS equation.
	Thus, for all $(u^\eps(t))\subseteq \Htwo\,,$ we have by Proposition~\ref{The inverse dynamical formula}\,,
	\be\label{explicit u-eps}
	u^\eps(t,z)=\ps{(\Id-z\eee^{-it}\eee^{-2itL_{u_0^\eps}}S^*)^{-1}\,u_0^\eps}{1}\,,\qquad \qlq \,z\in \D\,.
	\ee
	Our goal in this subsection is to pass to the limit in
	this formula. Therefore, we need first  to give a meaning to the operator
	$L_{u_0}$ when $u_0\in\Ltwo$\,.
	To handle this, we recall in a few lines the work of G\'erard--Lenzmann \cite[Appendix A]{GL22} who defined the operator $L_u$ with $u\in L^2_+(\R)$ via the standard theory of quadratic form. This new operator will coincide with the former Lax operator $L_u=D-T_{u}T_{\bar{u}}$ when $u\in H^2_+(\R)\,$. The same proof presented in \cite[Appendix A]{GL22}, works out on the torus $\T\,,$ and thus, one can define $L_u$ for $u\in\Ltwo\,.$ We recall the main points of the proof
	:
	\begin{enumerate}[(i)]
		\item For $u\in \Ltwo$ and $f,\,g\in H^{\frac12}_+(\T)\,,$  consider the quadratic form 
		$$
		\mathcal{Q}_u(f,g)=\ps{D^{1/2}f}{D^{1/2}g}-\ps{T_{\bar{u}}f}{T_{\bar{u}}g}\,.
		$$
		\item Observe that,  by decomposing $u$ in high and low frequency $$
		u(x)=u_N(x)+R_N(u,x)\,,\qquad 
		\begin{cases}
			u_N(x):=\sum_{n\geq 0}^N\fr{u}(n)\en[n]
			\\
			R_N(u,x):=\sum_{n\geq N+1}\fr{u}(n)\eee^{inx} 
		\end{cases}\,,
		$$ 
		and using Lemma~\ref{inegality Pi(u bar h)}\,, one can prove that   for all $\eta>0\,,$ $\e\, N_\eta:=N_\eta(u)\in\Nzero$ uniform on every compact set of $\Ltwo\,,$ such that 
		\be\label{perturbation}
		\Ldeux{T_{\bar{u}}f}^2<2\eta^2 \Big(\ps{Df}{f}+\Ldeux{f}^2\Big)+2N_\eta^2\|u\|_{L^2}^2\Ldeux{f}^2\,.
		\ee
		Therefore,
		\be\label{Qu f,f}
		\mathcal{Q}_u(f,f)\geq (1-2\eta^2)\|f\|_{{\dot{H}}^{1/2}}^2-2(N_\eta^2\|u\|_{L^2}^2+\eta^2)\Ldeux{f}^2\,.
		\ee
		\item
		Now, fixing $\eta$ small enough, there exists $K:=K(u)>0$  uniform on every compact of $\Ltwo\,,$ such that the following positive definite quadratic form  
		$$
		\tilde{\mathcal{Q}}_u(f,g):=\mathcal{Q}_u(f,g)+K\ps{f}{g}\,, \quad f,g\in H^{\frac12}_+(\T)\,, 
		$$
		 define a new inner product on $H^{\frac12}_+(\T)\,.$
		\item
		Using the theory of quadratic forms (see \cite{RS72}), 
		we introduce for $u\in\Ltwo$\,,
		\be
		\label{DomLu}
		\mathrm{Dom}(L_u)=\lracc{h\in H^{\frac12}_+(\T)\,;\ \e\, C>0\,,\ \va{\tilde{\mathcal{Q}}_u(h,g)}\leq C\Ldeux{g}, \qlq g\in  H^{\frac12}_+(\T)},
		\ee
		and    for any $f\in \mrm{Dom}(L_u)$
		\be
		\label{Lu, u L2}
		\ps{L_u(f)}{g}=\mathcal{Q}_u(f,g)\,, \quad \qlq g \in H^{\frac12}_+(\T)\,,
		\ee
		and one shows that this new operator $L_u$ is a \underline{self--adjoint} operator with a dense domain in $H^{\frac12}_+(\T)\,.$
	\end{enumerate}
	\begin{center}
		***
	\end{center}

	\subsubsection{Spectral properties of \texorpdfstring{$L_{u_0}$}{Lu0} for \texorpdfstring{$u_0\in \Ltwo$}{u0 in L2}}
	\label{uniqueness of the limit}
	Now that the operator $L_{u_0}$ has been introduced for $u_0\in \Ltwo\,,$ one can examine some of its spectral properties.
	As noted above, it is a self--adjoint operator with compact resolvent 
	%by Rellich--kondrachov theorem,
	then it has discrete spectrum. Moreover, its quadratic form $\mathcal Q_{u_0}$ is bounded from below. therefore,
	\be\label{spectre Lu0 u0 L2}
	\sigma(L_{u_0}):=\lracc{\la_0(u_0)\leq \ldots\leq \la_n(u_0)\leq \ldots}\,,\qquad \la_0>-\infty\,.
	\ee
	To characterize this spectrum, we use the following proposition.
	
	%Continuité lambda n
	\begin{prop}\label{Lipschitz}
		For every $n\in\Nzero\,$, the map $u\in L^2_+(\T)\mapsto \la_n(u) $ is Lipschitz continuous on compact subsets of  $L^{2}_{+}(\T)\,$.
	\end{prop}

	\begin{proof}
		Let $u\in\Ltwo\,.$ The key ingredient is to use the max--min principle,
		$$
		\la_{n}(u)= \max_{\underset{\dim F\,\leq n}{F\subseteq L^2_+}}\,
		\min\lracc{\mathcal{Q}_u(h,h)\,;\,h\in  F^{\perp}\cap H^{\frac12}_+(\T)\,,\ \|h\|_{L^2}=1 }.$$
    For any $v\in\Ltwo\,,$ $h\in H^{\frac12}_+(\T)\,,$ $\|h\|_{L^2}=1\,,$
    \begin{align*}
			\left| \mathcal{Q}_u(h,h) -\mathcal{Q}_v(h,h) \right| &=
			\left|\Ldeux{T_{\overline{v}}h}^2 -\Ldeux{T_{\overline{u}}h}^2 \right|\notag\\
			&\leq  \Ldeux{T_{(\overline{u}-\overline{v})}h}(\Ldeux{T_{\overline{v}}h}+\Ldeux{T_{\overline{u}}h})
            \\
            &\leq \Ldeux{u-v}\big(\,\Ldeux{u}+\Ldeux{v}\big)\left(1+\ps{Dh}{h}\right)
		\end{align*}
  thanks to inequality~\eqref{inegalite-Tu-bar-h}\,.
  Thus, 
  \be\label{Qv-Qu}
  \mathcal{Q}_v(h,h) \leq \mathcal{Q}_u(h,h) +   \Ldeux{u-v}\big(\,\Ldeux{u}+\Ldeux{v}\big)\left(1+\ps{Dh}{h}\right)\,.
  \ee
  In particular, considering any subspace $F$ of $\Ltwo$ of dimension $n$\,, and for any $h\in F^\perp\cap H^{\frac12}_+(\T) \cap \bigoplus_{k=0}^n\ker(L_u-\la_k(u)\Id)\,,$
the latter inequality holds.
In addition, observe by definition of $L_u=D-T_uT_{\bar u}\,,$ and by applying inequality~\eqref{perturbation} with $\eta=\frac12$\,,
  \begin{align*}
      \ps{Dh}{h}\leq &\, \ps{L_uh}{h}+\|T_{\bar u}h\|_{L^2}^2
      \\
      \leq &\, \ps{L_uh}{h} +\frac12 (\ps{Dh}{h}+1)+2N^2\|u\|_{L^2}^2\,, 
  \end{align*}
where $N\in\N$ is uniform on every compact subset of $\Ltwo$\,. That is,
\be\label{Dh|h}
     \ps{Dh}{h}\leq 2\la_n(u) +1+4N^2\|u\|_{L^2}^2\,,
\ee
since $\mathcal{Q}_u(h,h)=\ps{L_uh}{h}\leq \la_n(u)$ when $h\in \bigoplus_{k=0}^n\ker(L_u-\la_k(u)\Id)\,.$ Furthermore,
  applying once more the max-min principle, 
    \begin{align}\label{la[n]leq n}
        \la_n(u)
            \leq &\, \max_{\underset{\dim F\,\leq n}{F\subseteq L^2_+}}\,
		  \min\lracc{\ps{Dh}{h}\,;\,h\in          F^{\perp}\cap H^{\frac12}_+(\T)\,,\ \|h\|_{L^2}=1 }
    \\
        =&\, \min\lracc{\ps{Dh}{h}\,;\,h\in          \lracc{1,\ldots, \eee^{i(n-1)x}}^{\perp}\cap H^{\frac12}_+(\T)\,,\ \|h\|_{L^2}=1 }\notag
        \\
        = &\,n\notag
    \end{align}
Hence, combining \eqref{Qv-Qu}\,, \eqref{Dh|h} and \eqref{la[n]leq n}\,, we find for all $h\in F^\perp\cap H^{\frac12}_+(\T) \cap \bigoplus_{k=0}^n\ker(L_u-\la_k(u)\Id)\,,$ and since  $\mathcal{Q}_u(h,h)\leq \la_n(u)$\,,
\[
    \mathcal{Q}_v(h,h) \leq \la_n(u)+2\Ldeux{u-v}\big(\,\Ldeux{u}+\Ldeux{v}\big)\left(n +1+2N^2\|u\|_{L^2}^2\right)\,.
\]
  %  Therefore, $h=\sum_{k=0}^mc_kg_k(u)\,,$ $m\geq n\,,$ where the $(g_k(u))_{k=0}^m$ denotes any orthonormal basis of $\bigoplus_{k=0}^n\ker(L_u-\la_k(u)\Id)\,$, thus 
		% \be\label{Luhh}
		% \ps{L_uh}{h}\leq \sum_{k=0}^n|c_k|^2\la_k(u)\leq \la_n(u)\,.
		% \ee
Therefore, 
    $$
		\la_n(v)\leq \la_n(u)+2\,(n+1+2N^2\|u\|_{L^2}^2)\Ldeux{u-v}\left(\Ldeux{u}+\Ldeux{v}\right)\,.
    $$
\end{proof}

	\begin{corollary}[Characterization of the spectrum of  $L_{u_0}$]
		\label{caracterisation spectre Lu0}
		Let $(u_0^\eps)\subseteq H^2_+(\T)$ such that $u_0^\eps\to u_0$ in $\Ltwo$\,. The spectrum of $L_{u_0}$ is given by
		$$
		\sigma(L_{u_0})=\lracc{\,\lim_{\eps\to0}\la_n(u_0^\eps)\,\mid\; \la_n(u_0^\eps)\in \sigma(L_{u^\eps_0})\,,\, n\in\Nzero}\,.
		$$
	\end{corollary}

	\begin{proof}
		In light of the previous proposition, the result follows directly.
	\end{proof}
	
	\begin{prop}\label{resolvent cv}
		Let $u_0 \in L_{+}^2(\mathbb{T})$ and $\left(u_0^{\varepsilon}\right) \subseteq H_{+}^2(\mathbb{T})$ such that $u_0^{\varepsilon} \rightarrow u_0$ in $L_{+}^2(\mathbb{T})$. Then $L_{u_0^{\varepsilon}} \rightarrow L_{u_0}$ in the strong resolvent sense as $\varepsilon \rightarrow 0$\,.
	\end{prop}
	
	\begin{proof}
		% For $\lambda \ll 0\,$, let 
  For all $\eps>0$\,, we denote by $\phi_\la^{\eps}$ the vector $\phi_\la^{\varepsilon}:=\left(L_{u_0^{\varepsilon}}+\la\right)^{-1} h$\,, $h \in L_{+}^2(\mathbb{T})$\,, where $\la\gg 0\,.$
  Observe that $\la$ can be chosen uniformly with respect to $\eps$. 
  Indeed, by inequality~\eqref{perturbation}\,, and since $u_0^\eps\to u_0$ in $\Ltwo\,,$ there exists $N\in\Nzero$ uniform for all $\eps\,$, such that for $\eps$ small enough
  \be\label{coercive eps}
        \ps{(L_{\uzeroeps}+\lambda)g}{g}\geq \frac12\|g\|_{\dot H^\frac12}^2+(\la-\frac12-2N^2\|u_0\|_{L^2}^2)\|g\|_{L^2}^2\,.
  \ee
  Then, to apply the Lax--Miligram theorem, we choose in view of the last inequality $\la\gg 0$\,, such that $\ps{(L_{\uzeroeps}+\la)\cdot}{\cdot}$ is coercive and so $(L_{\uzeroeps}+\la)$ is invertible for all $\eps$ small. Our goal is to prove that $\phi^\eps_\la$ converges in $\Ltwo$. 
		For $g=\phi_\la^\eps$ in \eqref{coercive eps}\,, 
		\begin{align*}
	\ps{L_{\uzeroeps}\phi_\la^\eps}{\phi_\la^\eps}+\la\|\phi_\la^\eps\|_{L^2}^2
			\geq&\; \frac12 \| \phi_\la^\eps\|_{\dot H^{\frac12}}^2 +\big(\la-\frac12-2N^2\|u_0^\eps\|_{L^2}\big)\|\phi_\la^\eps\|_{L^2}^2 \,,
		\end{align*}
		which leads, for $\la\gg0\,$, to 
		$$
  \ps{L_{\uzeroeps}\phi_\la^\eps}{\phi_\la^\eps}+\la\|\phi_\la^\eps\|_{L^2}^2
		\gtrsim\;  \| \phi_\la^\eps\|_{H^{\frac12}}^2 \,.
		$$
		Hence, as $
		\ps{L_{\uzeroeps}\phi_\la^\eps}{\phi_\la^\eps}+\la\|\phi_\la^\eps\|_{L^2}^2=\ps{h}{\phi_\la^\eps}$\,,
		\be\label{control phi eps dans H1/2} 
		\ps{h}{\phi_\la^\eps}\gtrsim\; \, \| \phi_\la^\eps\|_{H^{\frac12}}^2 \,.
		\ee
  Using the Cauchy--Schwarz's inequality, we deduce $\| \phi_\la^\eps\|_{H^{\frac12}}^2 \leq \|h\|_{L^2}\|\phi_\la^\eps\|_{L^2}\,,$ where,   in view of Corollary~\ref{caracterisation spectre Lu0} and by \eqref{spectre Lu0 u0 L2}\,, we have for all $\eps>0\,,$ 
		\be
		\label{control L2 phi eps}
		\left\|\phi_\la^{\varepsilon}\right\|_{L^2} \leq \sup _n \frac{1}{\left|\lambda_n\left(u_0^{\varepsilon}\right)+\lambda\right|}\|h\|_{L^2} \leq C(\lambda)\|h\|_{L^2}\; .
		\ee    
		Therefore, 
  % applying Cauchy--Schwarz's inequality to \eqref{control phi eps dans H1/2}\,, we deduce by \eqref{control L2 phi eps}\,, 
		\[
		\| \phi_\la^\eps\|_{H^{\frac12}}^2\leq C(\la)\|h\|_{L^2}^2\;, \qquad \qlq \eps>0\,.
		\]
		Thus, there exists $\phi_\lambda \in L_{+}^2$ such that  up to a subsequence, 
		$$
		\phi_\la^{\varepsilon} \rightharpoonup \phi_\lambda \text { in }\Htwo
		\quad \text{ and } \quad
		\phi_\la^{\varepsilon} \rightarrow \phi_\lambda \text { in }\Ltwo\,.
		$$
		It remains to show that $\phi_\la=(L_{u_0}+\la)^{-1}h\,.$
		Indeed, for any $g\in H^{\frac12}_+(\T)\,,$ we have by definition of $\phi_\la^\eps\,,$ 
		$\left\langle\left(L_{u_0^{\varepsilon}}+\lambda\right) \phi_\la^{\varepsilon} \mid g\right\rangle=\langle h \mid g\rangle$. Namely,
		\begin{align}
			\label{Q u0 phi[lambda] g}
			\langle h \mid g\rangle
			% =\,&\, \mathcal{Q}_{u_0^\eps}(\phi_\la^\eps,g)-\la\ps{\phi_\la^\eps}{g}\notag
			% \\
			=\,&\,
			\langle D^{\frac12}\phi_\la^{\varepsilon} \mid D^{\frac12} g\rangle-\left\langle T_{\bar{u}_0^{\eps}} \phi_\la^{\eps} \mid T_{\bar{u}_0^{\eps}} g\right\rangle+\la\ps{\phi_\la^\eps}{g} . 
		\end{align}
		Since
		$$
		\begin{cases}
			T_{\bar{u}_0^{\varepsilon}} g \longrightarrow T_{\bar{u}_0} g \\
			T_{\bar{u}_0^\eps} \phi_\la^{\varepsilon} \rightharpoonup T_{\bar{u}_0} \phi_\la 
			\\
			D^{\frac12}\phi_\la^{\varepsilon} \rightharpoonup D^{\frac12}\phi_\lambda
		\end{cases}
		$$
		in $L_{+}^2(\mathbb{T})$ as $\varepsilon \rightarrow 0$\,, then passing to the limit in \eqref{Q u0 phi[lambda] g}, we infer for all $g \in H_{+}^\frac12(\mathbb{T})$,
		$$
		\langle h \mid g\rangle=
		\langle D^{\frac12}\phi_\la \mid D^{\frac12} g\rangle-\left\langle T_{\bar{u}_0^{\eps}g} \phi_\la^{\eps} \mid T_{\bar{u}_0^{\eps}}\right\rangle +\la\ps{\phi_\la}{g}
		\, =:\,
		\ps{(L_{u_0}+\la)\phi_\lambda}{g}. 
		$$
		That is, $\phi_\la\in \mrm{Dom}(L_{u_0})$ and $\phi_\la=(L_{u_0}+\la)^{-1}h\,.$ 
		Therefore, $L_{u_0^{\varepsilon}} \rightarrow L_{u_0}$ in the strong resolvent sense as $\varepsilon \rightarrow 0\,.$
	\end{proof}
	
	\subsubsection{Characterization of the limit \texorpdfstring{$u(t)$}{u(t)}}
	
	\begin{prop}[Uniqueness of the limit potential $u(t)$]
		\label{caracterisation u}
		Let $u_0\in L^2_+(\T)\,.$  
		There exists a unique potential 
		$ u(t)\in L^2_+(\T)$\,,
		\be\label{formule explicite limite u}
		u(t,z)=\ps{(\Id-z\eee^{-it}\eee^{-2itL_{u_0}}S^*)^{-1}\,u_0}{1}\,,\qquad \qlq \,z\in \D\,,
		\ee
		such that,
		for any sequence 
		$(u_0^\varepsilon)\subseteq H^2_+(\T)$ with
		$
		\|u_0^\eps-u_0\|_{L^2}\underset{\eps\to0}{\longrightarrow} 0\,,
		$
		we have
		$$
		u^\eps(t)\rightharpoonup u(t)  \text{ in }\Ltwo\,,\qquad \qlq t\in \R\,.
		$$
	\end{prop}

	\begin{proof}
		By the conservation of the $L^2$--norm (Lemma~\ref{loi de conservation}),  
  % for all $t\in\R\,,$
		$$
		\|u^\eps(t)\|_{L^2}=\|u_0^\eps\|_{L^2}\lesssim \|u_0\|_{L^2}\,, \qquad \qlq \eps\ll1\,.
		$$
		Then, $\qlq t\in\R\,,$ $\e\, u^*_t\in L^2_+(\T)$ such that   
		\be\label{u* in L2}
		u^{\eps}(t)\rightharpoonup  u^*_t \text{ in } \Ltwo\,,
		\qquad\text{ and }
		\qquad
		\|u_t^*\|_{L^2}\lesssim\|u_0\|_{L^2}\,.
		\ee
  Let \[
		u(t,z)=\ps{(\Id-z\eee^{-it}\eee^{-2itL_{u_0}}S^*)^{-1}\,u_0}{1}\,,\qquad \qlq \,z\in \D\,,
		\]
   and recall  by Proposition~\ref{The inverse dynamical formula}\,,
		\be\label{dynamical explicit u-eps}
		u^\eps(t,z)=\ps{(\Id-z\eee^{-it}\eee^{-2itL_{u_0^\eps}}S^*)^{-1}\,u_0^\eps}{1}\,,\qquad \qlq \,z\in \D\,.
		\ee
  Our goal is to prove that for all $t\in\R\,,$ $z\in \D\,,$ one has $u^\eps(t,z)\longrightarrow u(t,z)\,,$ and thus, by the uniqueness of the limit, one can conclude that $u^*_t$ is a well--defined function on $\Ltwo$,  given as a holomorphic function on $\D$ by $u(t,z)\,.$ 
		% In order to guarantee that the limit $u^*_t$ is indeed $u(t)\,,$ for a well-defined potential $u\,,$ we  use the  explicit formula. 
		% Indeed,
  Indeed,  by Proposition~\ref{resolvent cv}\,, $L_{\uzeroeps}\to L_{u_0}$ in the strong resolvent sense as $\uzeroeps\to u_0$ in $\Ltwo\,$. Thus, for any bounded continuous functions~$f$\,,  we have $f(L_{u_0^\eps})\to f(L_{u_0})$ in the strong operator topology  \cite[Proposition 10.1.9]{deO09}\,. In particular for all $t\in\R\,,$ and  for $f(x)=\eee^{-2ixt}$\,,
		$$
		\eee^{-2itL_{u_0^\eps}}\longrightarrow \eee^{-2itL_{u_0}}\,
		$$
		 strongly as $\eps\to0\,.$
		Therefore, passing to the limit in \eqref{dynamical explicit u-eps}\,, we deduce 
		\[
		u^\eps (t,z)\longrightarrow u(t,z)\,,\qquad \eps\to 0\,, \quad \qlq \,z\in \D\,,\, t\in\R\,.
		\]
	\end{proof}

	\subsection{ Strong convergence in \texorpdfstring{$\Ltwo$}{L2+} and conservation of the \texorpdfstring{$L^2$}{L2}-mass} 
 \label{Conservation of L2 mass}
 Our aim in this subsection is to prove Theorem~\ref{extension du flot a L2}\,.
	In light of the previous subsection, it remains to have
	\be
	\label{conservation mass L2}
	\|u(t)\|_{L^2}=\|u_0\|_{L^2}\,,
	\ee
	in order to guarantee the strong convergence of $u^\eps(t)\to u(t)$ in $L^2_+(\T)\,,$  as when $\eps\to 0\,,$
	\[
	\|u^\eps(t)\|_{L^2}=\|u^\eps_0\|_{L^2}\longrightarrow \|u_0\|_{L^2}\;.
	\]
	The main idea to prove~\eqref{conservation mass L2} is to use Parseval's identity on $u(t)$\,, where $u(t)$ is written in a suitable evolving  $L^2_+$--basis $(f_n^{\,t})\,,$ and satisfying 
	\be\label{égalité module birkhoff coord}
	\va{\ps{u(t)}{f_n^{\,t}}}=\va{\ps{u_0}{\fnzero}}\,, \qquad \qlq n\in\Nzero\,.
	\ee
	
	\begin{defi}[An orthonormal basis of $\Ltwo$]
		\label{definition-evolving orthonormal Basis}
		For all $\eps>0$\,, let $u^\eps\in \mathcal C(\R,\Htwo)\,.$ We denote by $(f_n^{\,\eps,t})$ the evolving orthonormal basis of $L^2_+(\T)$ along the curve $t\mapsto u^\eps(t)$ and satisfying the Cauchy problem 
		$$
		\begin{cases}
			\p_t \fnepst=B_{u^\eps(t)} \fnepst
			\\
			{\fnepst}_{|_{t=0}}=\fnzeroeps
		\end{cases}\,,
		$$
		for all $n$\,, where $(\fnzeroeps)$ is the orthonormal basis of $\Ltwo$ constituted from the eigenfunctions of the self--adjoint Lax operator $L_{u_0^\eps}\,,$ and $B_{u^\eps(t)}$ is the skew--adjoint operator defined in \eqref{Lax operators}\,.
	\end{defi}
	
	\begin{Rq}\label{fneps eigenvectors of Lueps}~
		\begin{enumerate}
			% \item As Proposition~\ref{Lu self-adjoint} ensure that the Lax operator $L_{\uzeroeps}$ is a self--adjoint operator with discrete spectrum, then 
			\item Since for all $\eps>0\,,$  $B_{u^\eps(t)}$ is a skew-adjoint bounded operator (cf. Proposition~\ref{Lu self-adjoint}) then the orthogonality of the $(\fnepst)$ is conserved in time. Indeed, for all $t\in\R\,,$
			$$
			\p_t\ps{\fnepst}{f_m^{\,\eps,t}}=\ps{B_u \fnepst}{f_m^{\,\eps,t}}+\ps{\fnepst}{B_u f_m^{\,\eps,t}}=0\,,
			$$
			\item By \cite[Lemma 4.1]{Ku06}\,, such orthonormal basis 
			% introduced in Definition~\ref{definition-evolving orthonormal Basis}\,, 
			is formed by the eigenfunctions of the Lax operator $L_{u^\eps(t)}$\,. Typically, we have for all $\eps>0\,,$  for all $t\in\R\,,$  
			$$
			L_{u^\eps(t)}\fnepst=\la_n(u_0^\eps)\fnepst\,.
			$$
		\end{enumerate}
	\end{Rq}
	
	\vskip0.3cm
	\noindent
	With this choice of $\Ltwo$--basis, we have a nice description of the evolution of the coordinates of $u^\eps(t)$\,. This is the aim of the next Lemma.
	
	\begin{lemma}\label{coordonnees-evolution}
		For all $\eps>0\,,$ let $u^\eps\in \mathcal C(\R, \Htwo)\,$ solution of \eqref{CS+}\,. Under the same notation of Definition~\ref{definition-evolving orthonormal Basis}\,, we have for any $n\in\Nzero$\,,
		\be
		\ps{u^\eps(t)}{\fnepst}
		=\ps{\uzeroeps}{\fnzeroeps}\,\eee^{-i\la_n(\uzeroeps)^2\,t}\,.
		\ee
	\end{lemma}

	\begin{proof}
		By Lemma~\ref{Reformulation de CS à l'aide des op. de Lax}\,, and since $L_{u^\eps(t)}$ and $B_{u^\eps(t)}$ are respectively self--adjoint and skew--adjoint operators
		\begin{align*}
			\p_t\ps{u^\eps(t)}{\fnepst}
			=&\,\ps{B_{u^\eps(t)}u^\eps(t) \,-i\,L^2_{u^\eps(t)}u^\eps(t)}{\fnepst}+\ps{{u^\eps(t)}}{B_{u^\eps(t)}\fnepst}
			\\
			=&\,-i\la_n^2(u_0^\eps)\,\ps{u^\eps(t)}{\fnepst}\,,
		\end{align*}
		which leads to the statement.
	\end{proof}
	
	\textsc{Consequence.} From the previous lemma, we infer for all $\eps>0$\,, $n\in\Nzero\,,$
	$$
	\va{\ps{u^\eps(t)}{\fnepst}}
	=\va{\ps{\uzeroeps}{\fnzeroeps}}\,.
	$$
	At this stage, we want to take $\eps\to0$ in the latter identity in order to deduce \eqref{égalité module birkhoff coord}\,. 
	However, one first might ask two questions :
	\begin{enumerate}[I.]
		\item Does the orthonormal basis $(\fnzeroeps)$ constituted from the eigenfunctions of the self--adjoint Lax operator $L_{u^\eps_0}$ remains an orthonormal basis of $\Ltwo$ under the limit $\eps\to0$\,?
		\item Suppose that the answer to the former question is affirmative, and denote  by $(\fnzero)$ this orthonormal basis limit. Based on Definition~\ref{definition-evolving orthonormal Basis}\,, 
		could we construct a time--evolving orthonormal basis, coinciding at $t=0$ with $(\fnzero)\,,$ and inducing a nice evolution as in Lemma~\ref{coordonnees-evolution} of the coordinates of $u$  in this basis ? 
		A priori, the operator $B_u$ defined in~\eqref{Lax operators}
		is not well--defined when $u\in \mathcal C_t[L^2_+(\T)]_x\,.$ Therefore, we should find another way to circumvent this problem.
	\end{enumerate}

	\vskip0.25cm
	The following proposition aims to answer question I. and to characterize the eigenfunctions of $L_{u_0}$ for $u_0\in\Ltwo$, by finding a uniform bound on the growth of the Sobolev norm $\|\fnzeroeps\|_{H^\frac12}$. For the second question II., we avoid the problem of defining $(\fnt)$ via Definition~\ref{definition-evolving orthonormal Basis} by using the same strategy done in the previous subsection,  that is, we
	% to 
	characterize the limit 
	% potential $u(t)$\,, but here we apply it on 
	$\fnt\,,$ for all $t\in\R\,.$ Therefore, we should derive an explicit formula of $\fnepst\,,$  for all $\eps>0$\, in order to pass to the limit. Unfortunately, we won't directly obtain that the limit $(\fnt)$ forms an orthonormal basis of $\Ltwo\,.$ However, it shall be an orthonormal family in $\Ltwo\,$, which will be sufficient to conclude.

	\begin{prop}
		\label{fn[0]}
		Given $u_0\in L^2_{+}(\T)$ there exists a sequence $(\fnzero)\subseteq \mrm{Dom}(L_{u_0})$,
		% define up to a phase, 
		such that for any sequence $(\uzeroeps) \subseteq H^2_+(\T)\,,$ 
		$
		u_0^\eps\to u_0 
		$ in $\Ltwo\,,$
		we have up to a subsequence
		$$
		\lim_{\eps\to0}\|\fnzeroeps-\fnzero\|_{L^2}= 0\,,
		\quad \qlq n\in\Nzero\,.
		$$ 
		% where $\fnzeroeps$ denotes the eigenfunction of $L_{u_0^\eps}\,.$
		In addition, for all $n\,,$
		$$
		L_{u_0}\fnzero=\la_n(u_0)\fnzero\,.
		$$
		
	\end{prop}
	
	\begin{proof}
		By definition of 
		$L_{\uzeroeps}=D-T_{\uzeroeps}T_{\bar{u}_0^\eps}\,,$ and
		since $L_{\uzeroeps}\, \fnzeroeps= \la_n(\uzeroeps)\fnzeroeps\,$,  it follows  
		$$
		\la_n(\uzeroeps)+\|T_{\bar{u}_0^\eps}\fnzeroeps\|_{L^2}^2=\|\fnzeroeps\|_{\dot{H}^{\frac12}_+}^2\,,\qquad \qlq n\geq0\,.
		$$
		Note that as $\uzeroeps\to u_0$ in $\Ltwo$\,, and by applying inequality~\eqref{perturbation}, we infer  that $\e\,N\geq 0$ independent of $\eps\,,$ such that
		$$
		\lambda_{n}(u_{0}^{\varepsilon})
		+\frac{1}{2}\,\|\fnzeroeps\|_{\dot{H}^{\frac{1}{2}}}^{2}+2N^2\|u_{0}^{\varepsilon}\|_{L^2}+\frac12>\|\fnzeroeps\,\|_{\dot{H}^{\frac12}}^{2}\,.
		$$
		Hence, by Proposition~\ref{Lipschitz} and since $\uzeroeps\to u_0$ in $\Ltwo$\,, 
		\be\label{borne to fn-uzeroeps-H1/2}
		\|\fnzeroeps\|_{\dot{H}^{1/2}}^2\,\lesssim \la_n(u_0)+\|u_0\|_{L^2}^2 \,.
		% \quad \qlq\, n\in \Nzero\,.
		\ee
		Therefore, up to a subsequence, $\e\, (f^0_n)$ such that, as $\eps\to0$\,,
		\be\label{fn[0,eps] cv dans H12 et L2}
		\fnzeroeps\rightharpoonup f^0_n\text { in }H^{\frac12}_+(\T)
		\qquad\text{ and }\qquad
		\fnzeroeps\to f^0_n \text{ in } L^2_+(\T)\,.
		\ee
		\vskip0.15cm
		
		At present, for the second part of the proof we show that the $(\fnzero)$ are eigenfunctions of $L_{u_0}\,.$
		% $(\fnzeroeps)$ converge weakly in $H^{\frac12}_+(\T)$\,, and by Rellich--Kandrachov's theorem, strongly  in $L^2_+(\T)$\,. We denote by $(\fnzero)$ the limit.
		Note that  by Lemma~\ref{inegality Pi(u bar h)}\,, one can directly check that $(f_n^0)\subseteq \mrm{Dom}(L_{u_0})$\, where $\mrm{Dom}(L_{u_0})$ was defined in \eqref{DomLu}. Besides, by definition of $L_{u_0^\eps}$\,, we have for all $g\in H^\frac12_+(\T)\,,$
		\be\label{ps Lu0-eps g}
		\langle D^\frac{1}{2}\fnzeroeps\mid D^\frac{1}{2} g\rangle-\ps{T_{\uzeroepsbar}\fnzeroeps}{T_{\uzeroepsbar}g} =\la_n(\uzeroeps)\ps{\fnzeroeps}{g}\,, 
		\ee
		where by Lemma~\ref{inegality Pi(u bar h)} $\ T_{\uzeroepsbar}g\longrightarrow T_{\bar u_0}g$ in $\Ltwo\,,$ by Proposition~\ref{Lipschitz} $\la_n(u_0^\eps)\to\la_n(u_0)\,,$ and by \eqref{fn[0,eps] cv dans H12 et L2} $\ T_{\uzeroepsbar}\fnzeroeps\rightharpoonup T_{\bar u_0}\fnzero\,.$
		Hence, passing to the limit in \eqref{ps Lu0-eps g}\,, we infer
		$$
		\ps{L_{u_0}\fnzero}{g}=\la_n(u_0)\ps{\fnzero}{g}\,,\quad\qlq g\in H^{\frac12}_+(\T)\,,
		$$
		leading to $L_{u_0}\fnzero=\la_n(u_0)\fnzero$ for all $n\geq 0\,,$ where $(\la_n(u_0))$ denotes all the spectrum of $L_{u_0}$ by Corollary~\ref{caracterisation spectre Lu0}\,.
		
	\end{proof}

	In the sequel, thanks  to Corollary~\ref{caracterisation spectre Lu0} and Proposition~\ref{fn[0]}\,,
	we denote by $(\fnzero)$ the orthonormal basis of $L^2_+(\T)$ made up of the eigenfunctions of $L_{u_0}$  obtained in the previous proposition.
	The following lemma aims to give an explicit formula to the $(\fnepst)$ defined in Definition~\ref{definition-evolving orthonormal Basis} in order to characterize at a second stage their limits when $\eps\to0 $.
	
	\begin{lemma}[The explicit formula of $\fnepst$]\label{explicit formula fneps}
		Under the same notation of Definition~\ref{definition-evolving orthonormal Basis}\,, we have for all $\eps>0\,,$ $t\in\R\,,$
		\be\label{dynamical explicit fnepst}
		\fnepst(z)=\ps{\Big(\Id-z\eee^{-it(L_{\uzeroeps}+\Id)^2}S^*\eee^{itL_{\uzeroeps}^2}\Big)^{-1}f_n^{\,\eps,0}}{\eee^{-itL^2_{\uzeroeps}}1}\,,\qquad \qlq \,z\in \D\,.
		\ee
	\end{lemma}
	
	\begin{proof}
		Like the proof of Proposition~\ref{The inverse dynamical formula}\,, we have
		$$
		\fnepst(z)=\ps{(\Id-zS^*)^{-1}\,\fnepst}{1},\qquad \qlq \,z\in \D\,.
		$$
		Using the unitary operator $U(t)$ introduced in \eqref{Cauchy problem unitary operator}\,, we deduce
		\begin{align}
			\fnepst(z)
			=&\,\left\langle U(t)^*\left(\Id-z \Sh^{*}\right)^{-1}\fnepst \mid U(t)^* 1\right\rangle
			\\
			=&\,\left\langle \,\left(\Id-z U(t)^*S^{*}U(t)\right)^{-1}\, U(t)^* \fnepst \mid U(t)^* 1\right\rangle\notag
			\\
			=&\,\left\langle \,\left(\Id-z U(t)^*S^{*}U(t)\right)^{-1} f_n^{\,\eps,0} \mid U(t)^* 1\right\rangle\,.\notag
		\end{align}
		By the formulae of \eqref{U(t)*1} and \eqref{U(t)*S*U(t)}\,, the explicit formula of $\fnepst$ follows.
	\end{proof}

	\begin{prop}\label{fn[t]}
		Let $u_0\in \B{1}\,.$ 
		Under the same notation of Definition~\ref{definition-evolving orthonormal Basis}\,, there exists an orthonormal family $(\fnt)$ of $\Ltwo$\,, such that for any sequence $(u_0^\eps)\subseteq H^2_+(\T)$\,, $u_0^\eps\to u_0$ in $\Ltwo$\,, we have up to a subsequence\,,
		$$
		\|\fnepst-\fnt\|_{L^2}\underset{\eps\to0}{\longrightarrow}0\,.
		$$
	\end{prop}
	
	\begin{proof}
		This proof is similar to the one done
		in Proposition~\ref{fn[0]}. However, it presents two main  differences. We will discuss these later in the upcoming remark. Now, coming back to the proof, recall by Proposition~\ref{caracterisation u}\,, there exists a unique $u(t)\in \Ltwo$ such that for any $\uzeroeps\to u_0$ in $L^2_+(\T)$ we have $u^\eps(t)\rightharpoonup u(t)$ in $L^2_+(\T)$ as $\eps\to 0\,.$
		Therefore, by definition of $L_{u^{\eps}(t)}=D-T_{u^{\eps}(t)}T_{\bar{u}^{\eps}(t)}$\,, and since $L_{u^{\eps}(t)}\fnepst=\la_n(u_0^\eps)\fnepst$ by the second point of Remark~\ref{fneps eigenvectors of Lueps}\,, 
		$$
		\la_n(\uzeroeps)+\|T_{\bar{u}^\eps(t)}\fnepst\|_{L^2}^2=\|\fnepst\|_{\dot{H}^{\frac12}}^2\;,  \qquad \qlq n\geq0\,.  
		$$
		Thus, applying Lemma~\ref{inegality Pi(u bar h)}\,,
		$$
		(1-\|u^\eps(t)\|_{L^2}^2)\,\|\fnepst\|_{\dot{H}^{\frac12}}
		\leq
		\|u^\eps(t)\|_{L^2}^2+\la_n(\uzeroeps)\,,\qquad\qlq n\geq 0\,.
		$$
		% Hence, we have
		Taking $\eps$ small enough to guarantee $\|u^\eps(t)\|_{L^2}=\|u^\eps_0\|_{L^2}<1\,,$  we deduce by Proposition~\ref{Lipschitz}\,, for $\eps$ small 
		\be\label{borne H12}
		\|\fnepst\|_{\dot{H}^{1/2}}^2\,\lesssim \frac{\la_n(u_0)+\|u_0\|_{L^2}^2}{1-\|u_0\|_{L^2}^2}\,,\qquad \qlq n\geq0\,.
		\ee
		% for all $n\,.$ 
		% since for $\eps$ small enough\,, $\|u^\eps(t)\|_{L^2}=\|u^\eps_0\|_{L^2}<1\,.$ 
		Hence, up to a subsequence,
		$$
		f_n^\eps(t)\rightharpoonup f_{n,\,t}^* \text{ in }H^{\frac{1}{2}}_+(\T)\,,
		\qquad
		f_n^\eps(t)\to f_{n,\,t}^* \text{ in }L^2_+(\T)\,.
		$$
		It remains to show that $f_{n,\,t}^*$ is uniquely characterized for all $t$. Using the  explicit formula of Lemma~\ref{explicit formula fneps}\,,
		$$
		\fnepst(z)=\ps{\Big(\Id-z\eee^{-it(L_{\uzeroeps}+\Id)^2}S^*\eee^{itL_{\uzeroeps}^2}\Big)^{-1}f_n^{\,\eps,0}}{\eee^{-itL^2_{\uzeroeps}}1}\,,\qquad \qlq \,z\in \D\,,
		$$ 
		and applying Proposition~\ref{resolvent cv}\,,
		% on $f_n^\eps(t)$\,,
		one can conclude that there exists 
		\be\label{fnt explicit}
		\fnt(z)=\ps{\Big(\Id-z\eee^{-it(L_{u_0}+\Id)^2}S^*\eee^{itL_{u_0}^2}\Big)^{-1}\fnzero}{\eee^{-itL^2_{u_0}}1}\,,\qquad \qlq \,z\in \D\,,
		\ee
		where $(\fnzero)$ denotes the eigenfunctions of $L_{u_0}$ obtained in Proposition~\ref{fn[0]}\,.
		Therefore, the limit   $f_{n,t}^*=\fnt$ for all $t$ on $\D$.
		Finally, observe that since the $(\fnepst)$ is an orthonormal basis of $\Ltwo$ and as $\fnepst\to \fnt$ in $\Ltwo\,,$ then  $(\fnt)$ forms an orthonormal family in $L^2_+(\T)$\,. 
	\end{proof}
	
	\begin{Rq}
		There are two main differences between the proof of Proposition~\ref{fn[0]} and Proposition~\ref{fn[t]}\,:
		\begin{enumerate}[(i)]
			\item  First, note that in the last proof, we cannot control the growth of the Sobolev norm $\|\fnepst\|_{H^{\frac{1}{2}}}$ uniformly for all $t$ by using  the inequality~\eqref{perturbation}\,,  since the integer $N_\eta$ in \eqref{perturbation} is not uniform for all $t\in\R\,.$ As an alternative, we rely on Lemma \ref{inegality Pi(u bar h)}\,. Consequently, the condition of $\|u^\eps(t)\|_{L^2}<1$  for $\eps$ small enough, is crucial here in order to conclude.
			\item Second in the previous proof, we had to give a meaning to the limit $f_{n,t}^*$ by characterizing this limit for all $t\in\R\,.$ 
		\end{enumerate}
		A common feature about these two proofs is to obtain a uniform bounds on the growth of the Sobolev norm $H^\frac12_+(\T)$ of the eigenfunctions $\fnzeroeps$ and $\fnepst$ to be able to conclude.
		
	\end{Rq}
	
	In view of Proposition~\ref{Lipschitz}\,, Lemma~\ref{coordonnees-evolution}\,, Proposition~\ref{fn[0]} and Proposition~\ref{fn[t]}\,,  we infer the following lemma.
	\begin{lemma}
		\label{ps u,fn}
		Let $u_0\in \B{1}\,.$ There exists an orthonormal \underline{family} $(f_n^t)
		$ of $ L^2_+(\T)$
		such that for all $n\geq 0\,,$ 
		% $L_{u(t)}f_n(t)=\l_nf_n(t)$ and
		\be\label{ps u,fn identite}
		\ps{u(t)}{\fnt}=\ps{u_0}{\fnzero}\eee^{-it\la_n^2(u_0)}\,, \quad\qlq t\in\R\,.
		\ee
	\end{lemma}
	
	\vskip0.25cm
	We are, at this stage, in a position to prove Theorem~\ref{extension du flot a L2}\,.
	
	\begin{theorem*}\ref{extension du flot a L2}\textbf{.}
		Let $u_0\in\B{1}$\,. There exists a unique potential $u\in\mathcal C(\R, L^2_+(\T))$ 
		such that,
		for any sequence $(u_0^\varepsilon)\subseteq H^2_+(\T)\,,$
		$
		\|u_0^\eps-u_0\|_{L^2}\underset{\eps\to0}{\longrightarrow} 0\,,
		$
		the following convergence holds~: for all $T>0\,,$
		$$
		\sup_{t\in[-T,T]}\|u^\eps(t)-u(t)\|_{L^2}\to 0\,, \quad \eps\to 0\,.
		$$
		In addition, 
		\be\label{dynamical explicit formula L2}
		u(t,z)=\ps{(\Id-z\eee^{-it}\eee^{-2itL_{u_0}}S^*)^{-1}\,u_0}{1}\,,\qquad \qlq \,z\in \D\,.
		\ee
		Moreover, the $L^2$--norm of the limit potential $u$ is conserved 
		\[
		\|u(t)\|_{L^2}=\|u_0\|_{L^2}\,,\quad \qlq t\in \R.
		\]
	\end{theorem*}
	
	\begin{proof}
		% {\bf Step 1.}
		Let $(t^\eps)\subseteq \R$ such that $t^\eps\to t$  as $\eps\to 0$. 
		Since $\|u_0^\eps-u_0\|_{L^2}\to 0$\,, then for $\eps$ small enough
		%     $$
		%         \|u^\eps(t)\|_{L^2}=\|u_0^\eps\|_{L^2}\lesssim \|u_0\|_{L^2}\,,
		%     $$ 
		% for all $t\in\R$\,.
		% In particular, for any sequence $(t^\eps)\subseteq \R$\,, such that $t^\eps\to t$  as $\eps\to 0$, we have
		$$
		\|u^\eps(t^\eps)\|_{L^2}=\|u_0^\eps\|_{L^2}\lesssim\|u_0\|_{L^2}\,.
		$$
		Hence, for any $t^\eps\to t$\,, there exists $ {u}_t^*\in L^2_+(\T)$ such that up to a subsequence,
		% $u^\eps(t^\eps)$ converges weakly in $ L^2_+(\T)\,$ to  $u_t^*$\,. Moreover, we have 
		$u^\eps(t^\eps)\rightharpoonup u_t^*$ in  $L^2_+(\T)$ and
		\be
		\label{borne-limite-inf}
		\|u_t^*\|_{L^2}\leq \liminf_{\eps\to0} \|u^\eps(t^\eps)\|_{L^2} = \liminf_{\eps\to0} \|u^\eps_0\|_{L^2} =\|u_0\|_{L^2}\,.
		\ee
		Our goal is to show that $u^\eps(t^\eps)$ converges strongly in $\Ltwo\,.$ As a first step, we stress that the weak limit potential
		$u_t^*$ is uniquely characterized for all $t$\,, and is equal to a unique limit $u(t)$\,. For that, we repeat the same proof of Proposition~\ref{caracterisation u} by
		% Note that we are under the same conditions of Proposition~\ref{caracterisation u}, then repeating the same proof of Proposition~\ref{caracterisation u} with
		exchanging $t$ into $t^\eps$ 
		% in the proof 
		with $t^\eps\to t$, 
		and we obtain $u^\eps(t^\eps)\rightharpoonup u(t)$ in $\Ltwo\,,$ where  
  $u(t)$ 
		% for a  potential $u$ 
		 is defined in equation~\eqref{formule explicite limite u}\,.
		% $\e\,! \,u(t)\in  L^2_+(\T)$ such that, for any $t^\eps\to t$\,, $t\in \R$
		% $$
		%     u^\eps(t^\eps)\rightharpoonup u(t)  \text{ in }\Ltwo\,.
		% $$
		Moreover,
		\be
		\label{leq}
		\|u(t)\|_{L^2}\leq \|u_0\|_{L^2}\,,
		\ee
		by \eqref{borne-limite-inf}. As a second step, we prove that this weak convergence in $\Ltwo$ is actually a strong convergence. 
		This can be achieved by checking  
		$$
		\Ldeux{u^\eps(t^\eps)}\to \Ldeux{u(t)}\,,\quad \eps\to 0\,.
		$$
		In fact, it is actually sufficient to prove that $\Ldeux{u(t)}=\Ldeux{u_0}$ since 
		%if it is the case then
		\be\label{cv en norme}
		\|u^\eps(t^\eps)\|_{L^2}=\|u^\eps_0\|_{L^2}\longrightarrow \|u_0\|_{L^2}
		% =\|u(t)\|_{L^2}
		\,, \quad \text{as } \eps\to 0\,.
		\ee
		Thanks to \eqref{leq}\,, we already have 
		$\|u(t)\|_{L^2}\leq \|u_0\|_{L^2}\,.$ Now, to prove $\|u(t)\|_{L^2}\geq \|u_0\|_{L^2}\,,$  we use Lemma~\ref{ps u,fn} 
		% --after exchanging $t$ with $t^\eps$ in the proof with $t^\eps\to t$\,--\; 
		to infer the existence of an orthonormal \textit{family} $(\fnt)$ of $L^2_+(\T)$ such that
		$$
		\sum_{n=0}^\infty \,\va{\ps{u(t)}{\fnt}}^2
		=\sum_{n=0}^\infty \,\va{\ps{u_0}{\fnzero}}^2
		=\|u_0\|_{L^2}^2\,.
		$$
		Hence, by Bessel's inequality 
		$$
		\|u(t)\|_{L^2}\geq \|u_0\|_{L^2}\,.
		$$
		As a conclusion, we have proved for any $t^\eps\to t$\,, $u^\eps(t^\eps)\to u(t)$ in $\Ltwo\,.$ This means, $u\in \mathcal C(\R,L^2_+(\T))$ and for all $T>0\,,$ 
		$$
		\sup_{t\in[-T,T]}\|u^\eps(t)-u(t)\|_{L^2}\to 0\,, \quad \eps\to 0\,.
		$$
	\end{proof}

	In view of the last Theorem, we denote through on, $u(t)$  the solution of \eqref{CS+} in $\Ltwo$ starting from an initial datum $u_0$ that lies inside the open ball $\mathcal{B}_{L^2_+}(1)$  of $L^2_+(\T)\,.$

	\begin{corollary}
		\label{ln[0]=ln[t]}
		The spectrum $\sigma(L_{u(t)})$ is invariant under the flow  of \eqref{CS+}\,.
	\end{corollary}

	\begin{proof}
		Let $(u_0^\eps)\subseteq H^2_+(\T)$ such that $ \|u_0^\eps-u_0\|_{L^2}\to0$ as $ \eps\to0\,.$
		Since, for all $n\in\Nzero\,,$
		% \be\label{ln[u0eps]}
		$$
		\la_n(u_0^\eps)=\la_n(u^\eps(t))\,,\quad \qlq t\in \R\,,
		$$
		% \ee
		% where $(u_0^\eps)\subseteq H^2_+(\T)$ such that $ \|u_0^\eps-u_0\|_{L^2}\to0$ as $ \eps\to0\,.$
		then by passing to the limit, we infer by Proposition~\ref{Lipschitz} and Theorem~\ref{extension du flot a L2} that the spectrum of $L_{u(t)}$ is conserved in time.
		
	\end{proof}

	\begin{corollary}\label{fn base + identite}
		Let $u_0\in \mathcal{B}_{L^2_+}(1)\,.$ There exists an orthonormal \underline{basis} $(f_n^t) $ of $L^2_+(\T)$ constituted from the eigenfunction of $L_{u(t)}$\,,
		such that for all $n\in \Nzero\,,$ 
		% $L_{u(t)}f_n(t)=\l_nf_n(t)$ and
		\be\label{coordonnees de u dans la base des vect.propre fn}
		\ps{u(t)}{\fnt}=\ps{u_0}{\fnzero}\eee^{-it \la_n^2(u_0)}\,, \quad\qlq t\in\R\,.
		\ee
	\end{corollary}

	\begin{proof}
		Taking into account Lemma~\ref{ps u,fn}\,, we only need to prove that the orthonormal family $(\fnt) $ found in Proposition~\ref{fn[t]} as
		\be\label{cv fn[eps]}
		\fnepst\rightharpoonup \fnt \text{ in }H^{\frac{1}{2}}_+(\T)\,,
		\qquad\text{ and } \qquad
		\fnepst\to \fnt \text{ in }L^2_+(\T)\,,
		\ee
		is actually an orthonormal basis of $L^2_+(\T)\,.$
		On the one hand, using \eqref{cv fn[eps]}  and since $u^\eps(t)\to u(t)$ in $\Ltwo\,,$
		% and the definition of $L_{u^\eps(t)}$ and $L_{u(t)}$ in \eqref{Lu, u L2}\,, 
		one can directly prove that $L_{u^\eps(t)}\fnepst\rightharpoonup L_{u(t)}\fnt$ in $\Ltwo\,$.
		On the other hand, using Proposition~\ref{Lipschitz}\,, we infer that taking $\eps\to 0$ in
		% For that, we prove that the $(f_n(t))$ describes all the eigenfunctions of the self--adjoint operator $L_{u(t)}$\,.
		% Thus, denoting by $u^\eps(t):=\mathcal S(t)u_0^\eps$ where  $(u_0^\eps)\subseteq H^2_+(\T)$ and 
		$$
		\ps{L_{u^\eps(t)}\fnepst}{g}=\la_n(u^\eps(t))\ps{\fnepst}{g}\,,\qquad \qlq g \in H^{\frac12}_+(\T)\,,
		$$
		% in order
		% to deduce that the $(f_n(t))$ describes all the eigenfunctions of the self--adjoint operator $L_{u(t)}$ and hence form an orthonormal basis of $L^2_+(\T)\,.$
		%  Thus,
		%which satisfies \eqref{ps u,fn identite}\,.
		leads to
		%   Therefore, by passing to the limit in the above identity and 
		%   using
		%   Therefore, by passing to the limit in the last identity and using 
		% Proposition~\ref{Lipschitz} and Corollary~\ref{ln[0]=ln[t]}\,, we find
		% we infer
		$$
		\ps{L_{u(t)}\,\fnt}{g}=\la_n(u(t))\ps{\fnt}{g}\,,\qquad \qlq g \in H^{\frac12}_+(\T)\,.
		$$
		As a result, the $(\fnt)$ describes all the eigenfunctions of the self--adjoint operator $L_{u(t)}$\,, thanks to Corollary~\ref{ln[0]=ln[t]} and Corollary~\ref{caracterisation spectre Lu0}\,. Hence, they form an orthonormal basis of $L^2_+(\T)\,.$
		% With $L_{u(t)}$ a self--adjoint operator, one can conclude that the orthonormal family $(f_n(t))$ is actually an orthonormal basis of $L^2_+(\T)\,.$
	\end{proof}
	\begin{Rq}
		The nice evolution in \eqref{coordonnees de u dans la base des vect.propre fn} of such coordinates  suggests that the so--called ``Birkhoff coordinates'' of \eqref{CS+} are the $(\ps{u(t)}{\fnt})$\,. 
    To be sure, we need to construct a one-by-one Birkhoff map $u\longleftrightarrow (\ps{u(t)}{\fnt})\,,$ similar to the remarkable achievement for the Benjamin--Ono equation in \cite{GK21}. 
    This construction can unlock several significant outcomes regarding the equation's dynamics.  In particular, the global well--posedness of the focusing equation when $\|u_0\|_{L^2}\geq 1\,.$
	\end{Rq}

	%%%%%%%%%%%%%%%%%%%%%%%%%%%%%%%%%%%%%%%%%%%%%%%%%
	%%%%%%%%%%%%%%%% Proof th 3 & 4  %%%%%%%%%%%%%%%
	%%%%%%%%%%%%%%%%%%%%%%%%%%%%%%%%%%%%%%%%%%%%%%%%%
	\vskip1cm
	\section{Proof of Corollary~\ref{GWP k geq 0} and Theorem~\ref{Compact} }
	\label{Proof of th 3 and 4}
	\vskip0.2cm
	
	To summarize, we have proved the global well--posedness of \eqref{CS+}--equation in $H^s_+(\T)$\,, $s>\frac32$\,, and for $s=0$ which correspond to $H^0_+(\T)\equiv L^2_+(\T)$\,. The following corollary aims to prove the global well--posedness for $0<s\leq \frac32\,.$
	\begin{figurehere}
		\begin{center}
			\begin{tikzpicture}
				\draw[->] (-0,0)--(13.25,0);
				\draw[red] (-0,0) node{\bf{I}} node[above=2.5]{$L^2_+$};
				\draw[red] (-0,-0.1) node[below] {Theorem~\ref{extension du flot a L2}};
				
				\draw[caputmortuum] (3.5,-0.1) node[below] {Corollary~\ref{GWP k geq 0}};
				% \draw[caputmortuum,-, line width=4,](0,-0)--(8,-0);
				\draw[caputmortuum,-, line width=4] (-0,0)--(6,0);
				\draw[caputmortuum] (6,0) node{{\bf I}};
				
				\draw[britishracinggreen] (6,0) node[above=2]{$H^\frac32_+$};
				\draw[britishracinggreen] (8.5,-0.1) node[below] {Theorem~\ref{GWP k geq 4}};
				\draw[britishracinggreen,->, line width=4,](6,-0)--(13.25,-0);
			\end{tikzpicture}
		\end{center}
	\end{figurehere}
	
	\vskip0.2cm
	\begin{corollary*}\ref{GWP k geq 0}\textbf{.}\label{Corollaire 1.3 section 4}
		For all $0\leq s\leq\frac32\,,$
		the Calogero--Sutherland DNLS focusing equation \eqref{CS+} is globally well--posed in $H^s_+(\T)\cap \mathcal{B}_{L^2_+}(1)$\,.   Moreover,  the following a--priori bound holds, 
		\[
		\sup_{t\in\R}\|u(t)\|_{H^s}\;\leq C\;,
		\]
		where $C=C(\|u_0\|_{H^s})>0$ is a positive constant.
	\end{corollary*}
	
	\begin{proof}
		For $s=0\,,$ we infer by Theorem~\ref{extension du flot a L2} the global well--posedness of the problem in $L^2_+(\T)$ in the sense of continuous extension of the flow from $H^2_+(\T)$ to $\Ltwo$\,.
		For $0<s\leq \frac32\,,$ let $u_0\in H^s_+(\T)\cap \mathcal{B}_{L^2_+}(1)$ and consider $(u^\eps_0)\subseteq H^2_+(\T)$ such that $u_0^\eps\to u_0$ in $H^s_+(\T)\,$. Then applying  Proposition~\ref{controle loi de conservation}\,, and since $u_0^\eps\to u_0$ in $H^s_+(\T)\,,$ there exists $C>0$ uniform with respect to all  small $\eps\,,$  such that 
		\[
		\frac{1}{C}\,\|u^\eps(t^\eps)\|_{H^s}
		\leq 
		\|(L_{u^\eps(t^\eps)}+\la)^s u^\eps(t^\eps)\|_{L^2}
		=\|(L_{u_0^\eps}+\la)^s u^\eps_0\|_{L^2}
		\leq C\|u^\eps_0\|_{H^s}\,,
		\]
		thanks to Lemma~\ref{loi de conservation}\,. Note that $\la$ is also uniform with respect to all $\eps$ small for the same reasons presented in the proof of Proposition~\ref{resolvent cv}\,. 
  Therefore, for all $\eps>0$ small,
		\be\label{u[eps] Hs}
		\|u^\eps(t^\eps)\|_{H^s}\leq C\|u_0\|_{H^s}\,.
		\ee
		Hence, as $t^\eps\to t$\,, we have $u^\eps(t^\eps)\rightharpoonup u(t)$ in $H^s_+(\T)$\,, where $u$ is a characterized function for all $t$ obtained as in Proposition~\ref{caracterisation u}\,. In particular, we infer $u^\eps(t^\eps)\to u(t)$ in $L^2_+(\T)$ with 
		\be\label{u Hs}
		\|u(t)\|_{H^s}\lesssim\|u_0\|_{H^s}\,,
		\ee
		by \eqref{u[eps] Hs}\,. As of now, to deduce the strong convergence in $H^s_+(\T)$ we use Proposition~\ref{controle loi de conservation}\,. 
		Thus, for all $\eps>0\,,$
		\begin{align}\label{u eps - u Hs}
			\|u^\eps(t^\eps)-u(t)\|_{H^s}^2
			\lesssim&\,
			\|(L_{u^\eps(t^\eps)}+\la)^s (u^\eps(t^\eps)-u(t))\|_{L^2}^2
			\\
			=&\,\|(L_{u^\eps(t^\eps)}+\la)^s u^\eps(t^\eps)\|_{L^2}^2
			+\|(L_{u^\eps(t^\eps)}+\la)^s u(t)\|_{L^2}^2\notag
			\\
			&\,-2\mrm{Re}\ps{(L_{u^\eps(t^\eps)}+\la)^s u^\eps(t^\eps)}{(L_{u^\eps(t^\eps)}+\la)^s u(t)}\notag
		\end{align}
		\vskip0.2cm
		\noindent
		Recall that $u^\eps(t^\eps)\to u(t)$ in $L^2_+(\T)$ which leads by Proposition~\ref{resolvent cv} to $L_{u^\eps(t^\eps)}\to L_{u(t)}$ in the strong resolvent sense. Thus, by functional calculus, (see the following lemma-- Lemma~\ref{troncature}) we infer 
		$$
		\begin{cases}
			(L_{u^\eps(t^\eps)}+\la)^s\, u(t) \to
			(L_{u(t)}+\la)^s \,u(t) \text{ in } \Ltwo\,,
			\\
			(L_{u^\eps(t^\eps)}+\la)^s\, \ueps(t^\eps) \rightharpoonup
			(L_{u(t)}+\la)^s\, u(t) \text{ in } \Ltwo\,,
		\end{cases}
		$$
		as $\eps\to0\,.$  In addition, for $\eps>0\,,$ recall by Lemma~\ref{loi de conservation}\,,
		\be\label{conservation premier terme}
		\|(L_{u^\eps(t^\eps)}+\la)^s\, u^\eps(t^\eps)\|_{L^2}^2
		=
		\|(L_{u^\eps_0}+\la)^s\, u^\eps_0\|_{L^2}^2\,.
		\ee
		Therefore, passing to the limit in \eqref{u eps - u Hs}\,, and since $u_0^\eps\to u_0$ in $H^s_+(\T)$ combined with Lemma~\ref{troncature} and Proposition~\ref{controle loi de conservation}\,, we deduce for all $\eps$ small,
		$$
		\|u^\eps(t^\eps)-u(t)\|_{H^s}^2
		\lesssim
		\|(L_{u_0}+\la)^s u_0\|_{L^2}^2
		-\|(L_{u(t)}+\la)^s u(t)\|_{L^2}^2\,.
		$$
		At this stage, it remains to show that the right--hand side of the previous inequality is vanishing. Indeed, by Corollary~\ref{ln[0]=ln[t]}\,,
		\begin{align*}
			\ps{(L_{u(t)}+\la)^{2s}u(t)}{u(t)}
			=&\,\sum_{n\geq 0}(\la_n(u(t))+\la)^{2s}\va{\ps{u(t)}{\fnt}}^2
			\\
			=&\,\sum_{n\geq 0}(\la_n(u_0)+\la)^{2s}\va{\ps{u_0}{\fnzero}}^2
			=\,\ps{(L_{u_0}+\la)^{2s}u_0}{u_0}\,,
		\end{align*}
		where $(\fnt)$ are the orthonormal basis obtained in Corollary~\ref{fn base + identite} since $u^\eps(t^\eps)\to u(t)$ in $L^2_+(\T)\,$.
		As a result, $\|(L_{u_0}+\la)^s u_0\|_{L^2}^2
		=\|(L_{u(t)}+\la)^s u(t)\|_{L^2}^2$\,, and
		as $\eps\to 0$\,, 
		\[
		\|u^\eps(t^\eps)-u(t)\|_{H^s}^2
		\longrightarrow0\,.
		\]
		Hence, $u\in\mathcal{C}(\R,H^s_+(\T))$ such that \eqref{u Hs} is satisfied, and for all $T>0\,,$
		\[
		\sup_{t\in[-T,T]}\|u^\eps(t)-u(t)\|_{H^s}\to 0\,.
		\]
	\end{proof}

	% \begin{Rq}
	% 	\begin{enumerate}
	% 		% \item  Note for $0<s\leq \frac32$\,, the condition of $u_0\in \B{1}$ is important just to infer the existence of $u^\eps\in \mathcal{C}(\R, H^2_+(\T))$ for $\eps$ small enough, in the last proof. However, at this stage one can apply Proposition~\ref{fn[t]} without requiring $u_0\in \B{1}$ since  $u^\eps(t^\eps)\to u$ in $\Ltwo$ leads to the fact that we can use the perturbation inequality~\eqref{perturbation} instead of Lemma~\ref{inegality Pi(u bar h)} to obtain a uniform control of the growth of the Sobolev norm $\|\fnepsteps\|_{H^\frac12}$ in inequality~\eqref{borne H12}\,.
	% 		\item The inequality~\eqref{u Hs} in the  proof implies that the flow 
	% 		$$
	% 		u_0\in H^s_+(\T)\cap \mathcal{B}_{L^2_+}(1) \mapsto u\in\mathcal{C}(\R\,,H^s_+(\T))\,,
	% 		$$
	% 		defined in the previous corollary, is continuous.
	% 		\item \label{extension loi de conservation} As observed in the proof, one has  for any $u\in\mathcal C(\R, H^r_+(\T))$\,,  $r\geq0$ such that $\|u\|_{L^2}<1$\,,
	% 		the family $\lracc{\mathcal{H}_s(u):=\ps{(L_u+\la)^su}{u}\,;\, 0\leq s\leq 2r}$\,, $\la>\!>0\,,$ is conserved along the flow. This  means, that the conservation laws of Lemma~\ref{loi de conservation} are extended to less regular $u$\,.
	% 		Consequently, as noted in Remark~\ref{r>3/2}\,, Proposition~\ref{controle loi de conservation} holds for any $u_0\in \B{1}\cap H^r_+(\T)\,,$ for all $r\geq 0\,.$ 
	% 	\end{enumerate}
		
	% \end{Rq}
	To conclude the proof of Corollary~\ref{GWP k geq 0}, we need to prove the following functional analysis result.
	\begin{lemma}\label{troncature}
		Let $(A_\eps)$ be a sequence of  positive self--adjoint operators in $L^2$\,. Suppose that  $A_\eps\to A$ in the strong resolvent sense as $\eps\to 0\,,$ and for all $s\geq 0\,,$
		$$
		\mrm{Dom}(A_\eps^s)=\mrm{Dom}(A^s)=H^s\,, \qquad \eps>0\,.
		$$ 
		Moreover, assume that for all $u\in H^s\,,$  the $(A_\eps^su)$ are  uniformly bounded with respect to $\eps>0$ in the following sense $\|A_\eps^su\|\leq C\|u\|_{H^s}\,.$
		% Let $(u^\eps)$ be a bounded sequence in $\mrm{Dom}(A_\eps^s)$\,, i.e. $\|A_\eps^su^\eps\|\leq C$\,, such that $u^\eps\to u$ in $\mathscr{H}$\,. 
		Then, for all $s\geq0\,,$
		\be\label{Aeps[s] u to A[s]u}
		A_\eps^su\longrightarrow A^su \text{ in } L^2\,, \qquad \eps\to 0\,.
		\ee
		
	\end{lemma}
	
	\begin{proof}
		For all $R>0$, let $\chi_R\in \mathcal{C}^\infty(\R_+)$ such that $\chi_R\equiv 1$ on $[0,R]$ and $\mrm{supp}(\chi_R)\subseteq[0,2R]$\,. Note that, for all $s\geq 0\,,$ the subset $\lracc{\chi_R(A)\,u\;;\; u\in H^s, \; R>0}$ is dense in $H^s\,.$ 
		Then, 
		since the $(A_\eps^su)$ are uniformly bounded with respect to $\eps$\,,
		it is sufficient to prove for all $R>0\,,$
		% for all $\eta>0$\,, $\e \eps_0>0$ such that $\qlq \eps>\eps_0\,,$
		$$
		A_\eps^s\, \chi_R(A)u\longrightarrow A^s \chi_R(A)u \,,\qquad\eps\to0\,, 
		$$
		to obtain \eqref{Aeps[s] u to A[s]u}\,.
		Toward this end, let $R>0$, and write for any $s\geq 0\,,$ for all $\eps>0\,,$
		\be\label{decomposition suivant chi[tilde r]}
		A_\eps^s\, \chi_R(A)u = A_\eps^s\,\chi_{\tilde{R}}(A_\eps) \chi_R(A)u 
		+ A_\eps^s\,(1-\chi_{\tilde{R}}(A_\eps)) \chi_R(A)u\,,
		\ee
		where  $\chi_{\tilde{R}}\in \mathcal{C}^\infty(\R_+)$ such that $\chi_{\tilde{R}}\equiv 1$ on $[0,\tilde{R}]$ and $\mrm{supp}(\chi_{\tilde{R}})\subseteq[0,2\tilde{R}]$\,, $\tilde{R}\geq 2R\,$. Notice that, 
		\begin{align*}
			\| A_\eps^s\,(1-\chi_{\tilde{R}}(A_\eps)) \chi_R(A)u\|
			=&\, \| A_\eps^{-s}\,(1-\chi_{\tilde{R}}(A_\eps)) A_\eps^{2s}\chi_R(A)u\|
			\\
			\leq &\,\frac{C^2}{(2\tilde{R})^{s}}\,\|u\|_{H^s}\,.
		\end{align*}
		Therefore, for all $\eta>0\,,$ there exists $\tilde{R}>\!>0\,,$ such that  for all $\eps>0\,,$
		$$
		\| A_\eps^s\,(1-\chi_{\tilde{R}}(A_\eps)) \chi_R(A)u\|< \eta\,,
		$$
		and so, by \eqref{decomposition suivant chi[tilde r]}\,,
		\be\label{R tilde grand}
		\| A_\eps^s\, \chi_R(A)u-A_\eps^s\,\chi_{\tilde{R}}(A_\eps) \chi_R(A)u \|<\eta\,, \qquad \qlq\eps>0\,.
		\ee
		Besides, recall that $A_\eps\to A$ in the strong resolvent sense as $\eps\to 0\,.$ Hence, by \cite[Proposition 10.1.9]{deO09}\,, $f(A_\eps)\to f(A)$ as $\eps\to0$\,, in the  operator norm for all continuous bounded $f$\,. In particular, for $f(x)=x^s\, \chi_{\tilde{R}}(x)\,,$ we have
		\be \label{chi R[tilde] A[eps] cv}
		A_\eps^s\:\chi_{\tilde{R}}(A_\eps)\,\chi_R(A)u\longrightarrow A^s\,\chi_R(A)u\,, \qquad \eps \to 0\,.
		\ee
		Thus, combining \eqref{R tilde grand} and \eqref{chi R[tilde] A[eps] cv}\,, we infer for all $\tilde{\eta}>0\,,$ there exists $\eps_0>0$ such that $\qlq \eps>\eps_0\,,$
		$$
		\| A_\eps^s\, \chi_R(A)u- A^s\, \chi_R(A)u\|<\tilde{\eta}\,.
		$$
	\end{proof}

	\vskip0.2cm
	Beyond the global well--posedness results of the \eqref{CS+} Cauchy's Problem, we are interested in some qualitative properties about the flow $\mathcal S^+(t)$ of this equation. Therefore, we prove that all weak limit points of the orbit are actually strong limit points. 
	
	\begin{theorem*}\ref{Compact} \textbf{.}
		Given an initial data $u_0\in \B{1}\cap H^s_+(\T)$\,, $s\geq 0\,,$ the orbit of the solution $\lracc{\mathcal{S}^+(t)u_0\,;\,t\in\R}$ is relatively compact in $H^s_+(\T)\,.$
	\end{theorem*}

	\begin{proof}  Let    $(t_n)\subseteq~\R$ such that $t_n\to\infty\,.$
		\vskip0.1cm
		\noindent
		\underline{\textbf{Step 1} : $s=0\,$. }
		By  Theorem~\ref{extension du flot a L2}\,,
		\be\label{utn=u0 norme L2}
		\|u(t_n)\|_{L^2}=\|u_0\|_{L^2}\,.
		\ee
		Then, $\e\, \tilde{u}\in L^2_+(\T)$ such that, up to a subsequence, 
		$$
		u(t_n)\rightharpoonup\tilde{u} \text{ in } \Ltwo
		\qquad \text{ and } \qquad 
		\|\tilde{u}\|_{L^2}\leq \|u_0\|_{L^2}\,.
		$$
		In order to obtain the strong convergence $u(t_n)\to \tilde{u}$ in $\Ltwo\,,$ all it remains is to show that
		$
		\|u(t_n)\|_{L^2}\to \|\tilde{u}\|_{L^2} 
		$\,, or by \eqref{utn=u0 norme L2}\,, $$\|u_0\|_{L^2}=\|\tilde{u}\|_{L^2}\,.$$
		Observe that we already have  $\|u_0\|_{L^2}\geq\|\tilde{u}\|_{L^2}$\,. 
		Now, to prove  $\|u_0\|_{L^2}\leq\|\tilde{u}\|_{L^2}\,,$ recall by Corollary~\ref{fn base + identite}\,, 
		% we know that
		\be\label{coordonnée de u dans fn(tn)}
		\ps{u(t_n)}{f_m^{\,t_n}}=\ps{u_0}{f^{\,0}_m}\eee^{-it_n\,\la_m^2(u_0)\,}\,, 
		\ee
		where $(f_m^{\,t_n})$ is the orthonormal basis of $\Ltwo$ constituted of the eigenfunctions of $L_{u(t^n)}\,.$
		The idea is to pass to the limit as $t_n\to\infty$ in the above identity and conclude by using Bessel's identity.
		First, we have $u(t_n)\rightharpoonup\tilde{u}$ in $\Ltwo$\,. Second, notice that the $(f_m^{\,t_n})$ converges strongly in $\Ltwo$ as $t_n\to\infty$ to an orthonormal family denoted by $(g_m)$\,. Indeed,  by definition of $L_{u(t_n)}$\,,
		$$
		\la_m(u_0)+\|T_{\,\overline{u(t_n)}}\,f_m^{\,t_n}\|_{L^2}^2=\|f_m^{\,t_n}\|_{\dot{H}^{\frac{1}{2}}}^2\,,
		$$
		leading  to $\|f_m^{\,t_n}\|_{H^{\frac{1}{2}}}\lesssim\frac{\la_m(u_0)+\|u_0\|_{L^2}^2}{1-\|u_0\|_{L^2}^2}$ for all $m$\,, thanks to Lemma~\ref{inegality Pi(u bar h)}\,. 
		Hence, by Rellich--Kondrachov's Theorem\,, $f_m(t_n)\to g_m$ in $\Ltwo$ as $t_n\to\infty\,.$  
		Third, using Cantor diagonalization procedure, one can extract a subsequence $t_n\to\theta_m\;\mrm{mod}(\frac{2\pi}{\la_m^2})$\,, as the circle is compact. 
		Hence, by passing to the limit in \eqref{coordonnée de u dans fn(tn)}, we obtain
		$$
		\ps{\tilde{u}}{g_m}=\ps{u_0}{f^{\,0}_m}\eee^{-i\la_m^2(u_0)\theta_m}\,.
		$$
		% where by Cantor diagonalization procedure one can extract a subsequence $t_n$ converging to a certain $\theta_m\;\mrm{mod}(\frac{2\pi}{\la_m^2})$\,, since the circle is compact. Hence, by 
		As a result, using Bessel's inequality, we conclude 
		$$
		\|\tilde{u}\|_{L^2}^2
		\geq
		\sum_{n=0}^\infty\va{\ps{\tilde{u}}{g_n}}^2
		=
		\sum_{n=0}^\infty\va{\ps{u_0}{f_m^0}}^2
		=
		\|u_0\|_{L^2}^2\,.
		$$
		Consequently, $ \|u(t_n)\|_{L^2}\to \|\tilde{u}\|_{L^2} $ and thus $u(t_n)\to \tilde{u}$ in $\Ltwo\,.$
		\vskip0.1cm
		\noindent
		\underline{\textbf{Step 2} : $s>0\,$. } By inequality~\eqref{u Hs} of the proof of Corollary~\ref{GWP k geq 0}\,, we have $\|u(t_n)\|_{H^s}\lesssim \|u_0\|_{H^s}$ leading to 
		\[
		u(t_n)\rightharpoonup \tilde{u} \text{ in } H^s_+(\T)
		\qquad \text{ and }\qquad
		\|\tilde{u}\|_{H^s}\lesssim \|u_0\|_{H^s}\,.
		\]
		In particular, $u(t_n)\to \tilde{u}$ in $L^2_+(\T)$\,. Then, in view of Remark~\ref{r>3/2}\,,
		\begin{align}\label{u(tn)-tilde[u]}
			\|u(t_n)-\tilde{u}\|_{H^s}^2
			\lesssim &\,\|(L_{u(t_n)}+\la\Id)^s(u(t_n)-\tilde{u})\|_{L^2}^2
			\\
			=&\,\|(L_{u(t_n)}+\la\Id)^s \,u(t_n)\|_{L^2}^2
			+\|(L_{u(t_n)}+\la\Id)^s \,\tilde{u}\|_{L^2}^2\notag
			\\
			&\,-2\mrm{Re}\ps{(L_{u(t_n)}+\la)^s\, u(t_n)}{(L_{u(t_n)}+\la)^s\, \tilde{u}}\,,\notag 
		\end{align}
		where by the second point of Remark~\ref{extension conservation laws}\,, 
		\begin{align*}
			\|(L_{u(t_n)}+\la\Id)^s \,u(t_n)\|_{L^2}^2=
			\,\|(L_{u_0}+\la)^s \,u_0\|_{L^2}^2\,.
		\end{align*}
		Besides, as $u(t_n)\to \tilde{u}$ in $\Ltwo$\,, then  $L_{u(t_n)}\to L_{\tilde{u}}$ in the strong resolvent sense thanks to  Proposition~\ref{resolvent cv}\,. Hence, by functional calculus (see Lemma~\ref{troncature}), we infer 
		$$
		\begin{cases}
			(L_{u(t_n)}+\la)^s\, \tilde{u} \,\to
			\, (L_{\tilde{u}}+\la)^s\,\tilde{u}\; \text{ in }\, \Ltwo\,,
			\\
			(L_{u(t_n)}+\la)^s\, u(t_n) \rightharpoonup
			(L_{\tilde{u}}+\la)^s \,\tilde{u}\, \text{ in } \Ltwo\,,
		\end{cases}
		$$
		as $n\to\infty\,.$  Therefore, by passing to the limit in \eqref{u(tn)-tilde[u]}, 
		we deduce
		$$
		\lim_{n\to+\infty}\|u(t_n)-\tilde{u}\|_{H^s}^2
		\lesssim
		\|(L_{u_0}+\la)^s u_0\|_{L^2}^2
		-\|(L_{\tilde{u}}+\la)^s \tilde{u}\|_{L^2}^2\,,
		$$ where the right--hand side vanishes. Indeed, by Corollary~\ref{ln[0]=ln[t]} and Proposition~\ref{Lipschitz}\,, $\la_n(\tilde{u})=\la_n(u_0)$\,. Hence,
		\begin{align*}
			\ps{(L_{\tilde{u}}+\la)^{2s}\tilde{u}}{\tilde{u}}
			=&\,\sum_{n\geq 0}(\la_n(u_0)+\la)^{2s}\va{\ps{\tilde{u}}{g_n}}^2
			=\sum_{n\geq 0}(\la_n(u_0)+\la)^{2s}\va{\ps{u_0}{\fnzero}}^2
			\\
			=&\,\ps{(L_{u_0}+\la)^{2s}u_0}{u_0}\,,
		\end{align*}
		where $(g_n)$ is the orthonormal family of $\Ltwo$ found in Step 1. Nevertheless, since $u(t_n)\to\tilde{u}$ in $\Ltwo\,,$ one could show as in Corollary~\ref{fn base + identite} that this orthonormal family is indeed an orthonormal basis of $\Ltwo$ by proving that the $(g_n)$ constitutes all the eigenfunction of the self--adjoint operator $L_{\tilde{u}}\,.$ As a consequence,
		$$
		\|(L_{u_0}+\la)^s\, u_0\|_{L^2}^2
		=\|(L_{\tilde{u}}+\la)^s\, \tilde{u}\|_{L^2}^2\,,
		$$
		and thus $\|u(t_n)-\tilde{u}\|_{H^s}^2\to 0\,,$ as $n\to\infty\,.$
	\end{proof}
	~\\
	%%%%%%%%%%%%%%%%%%%%%%%%%%%%%%%%%%%%%%%%%%%%%%%%%%%%%%%%
	%%%%%%%%%%%%%%%%%%%%%% Defocusing eqt %%%%%%%%%%%%%%%%%%%
	%%%%%%%%%%%%%%%%%%%%%%%%%%%%%%%%%%%%%%%%%%%%%%%%%%%%%%%%

	% \vskip1cm
	\section{The Calogero--Sutherland DNLS defocusing equation (CS\texorpdfstring{$^-$}{-})}
	\label{defocusing}
	
	\vskip0.2cm
	In this section, we consider the defocusing equation of \eqref{CS} 
	\be\tag{CS$^-$}\label{CS-}
	i\partial_tu+\partial_x^2u-2D_+(|u|^2)u=0\,.
	\ee
	Note that by adapting the argument of \cite[Proposition 2.1]{GL22} to the defocusing equation, one can infer the local well--posedness of the \eqref{CS-} problem in $H^s_+(\T)$ for $s>\frac{3}{2}\,$. And we expect that one can go down to $s>\frac12$ by following \cite{deMP10}\,.
	~\\
	
	\textit{Below are a series of lemmas, propositions, and theorems  that can be proved similarly to their analogs in the focusing case. }
	Again, the integrable methods are the main ingredients to conclude. The first proposition is to announce that the defocusing equation of \eqref{CS-} enjoys also a Lax pair formalism. 
	
	\begin{prop}[Lax pair for \eqref{CS-}]
		Let $u\in\mathcal{C}([-T,T]\,,H^s_+(\T))\,,$ $s>\frac{3}{2}\,,$ be a solution of \eqref{CS-}\,. There exist two operators  
		$$
		\tilde{L}_u=D \,\textcolor{red}{+}\, T_uT_{\overline u}\,, \qquad\tilde{B}_u=\,\textcolor{red}{-}T_uT_{\partial_x\overline u}\,\textcolor{red}{+}\,T_{\partial_xu}T_{\overline u} +i(T_uT_{\overline u})^2
		$$
		satisfying  the Lax equation 
		$$
		\frac{d\tilde{L}_u}{dt}=[\tilde{B}_u,\tilde{L}_u]\,.
		$$ 
	\end{prop}
	
	\begin{lemma}
		Given $u\in \mathcal C([-T,T],H^r_+(\T))\,,$ $r>\frac32\,,$ a solution of \eqref{CS-} equation, then 
		$$
		\p_t u=\tilde{B}_u u -i\tilde{L}^2_uu\,.
		$$
		As a consequence, the quantities $\tilde{\mathcal H}_s(u):=\langle(\tilde{L}_u+\lambda)^su\mid u\rangle\,,$ $\lambda>0\,,$ are  conserved by the flow $\mathcal S^-(t)$ of  \eqref{CS-} for all $0\leq s\leq 2r\,.$
	\end{lemma}

	\begin{Rq}
		Expanding the conservation laws $\tilde{\mathcal{H}}_k(u)$ for all $k\in\Nzero\,,$ we have
		\begin{align*}
			\tilde{\mathcal H}_0(u)&=\langle u(t) \mid u(t) \rangle=\|u(t)\|_{L^2}^2=\|u_0\|_{L^2}
			\\
			\tilde{\mathcal H}_1(u)&=\langle \tilde{L}_{u(t)} u(t)
			\mid u(t) \rangle=\|u(t)\|_{\dot{H}^{1/2}}^2+\|T_{\bar u(t)} u(t)\|_{L^2}^2\geq \|u(t)\|_{\dot{H}^{\frac12}}^2
			\\ 
			\tilde{\mathcal H}_2(u)&=\|\tilde{L}_{u(t)}u(t)\|_{L^2}^2
			\geq(1-\eps)\|Du(t)\|_{L^2}^2 +(1-C_{\eps})\|T_{u(t)}T_{\bar u(t)}u(t)\|_{L^2}^2\geq \|u(t)\|_{\dot{H}^1}^2-C(\|u\|_{H^{\frac12}})\,,
			% \\
			% \tilde{\mathcal H}_4(u)&=\|\tilde{L}^2_{u(t)}u(t)\|_{L^2}^2=\|D\tilde{L}_{u(t)}u(t) +T_{u(t)}T_{\bar u(t)}(\tilde{L}_{u(t)}u(t))\|_{L^2}^2\,,
			\\
			&\hskip4cm\vdots
		\end{align*}
  thanks to Young's and Sobolev's inequalities.
		Unlike the focusing case, here
		we deduce the uniform control of the growth of Sobolev norms of the solution $u$ by the conservation laws,
		without requiring any additional condition of smallness on the initial data $u_0\,.$  Therefore, Proposition~\ref{controle loi de conservation} holds in the defocusing case for all $u_0\in H^r_+(\T)\,,$ $r>\frac32$.
	\end{Rq}
	
	As a result, we state the following theorem  which is the analog of Theorem~\ref{GWP k geq 4} but for equation~\eqref{CS-}.
	
	\begin{theorem} 
		For all $s>\frac32\,,$
		let $u_0\in H^s_+(\T)$\,. There exists  a unique global solution $u\in \mathcal C(\R,H^s_+(\T))$ of the defocusing equation \eqref{CS-}, satisfying at $t=0$, $u(0,\cdot)=u_0\,.$ Furthermore, for all $s>\frac32\,,$
		\[
		\sup_{t\in\R}\|u(t)\|_{H^s}\;\leq C\;,
		\]
		where $C=C(\|u_0\|_{H^s})>0$ is a positive constant.
	\end{theorem}
	
	As for the focusing case, the defocusing Calogero--Sutherland DNLS has an explicit solution.
	
	\begin{lemma}
		Let $u_0\in\mathcal{C}(\R,H^s_+(\T))\,,$ $s>\frac32$ then the solution of the defocusing Calogero--Sutherland DNLS equation~\eqref{CS-defocusing} is given by
		\[
		u(t,z)=\ps{(\Id-z\eee^{-it}\eee^{-2it\tilde{L}_{u_0}}S^*)^{-1}\,u_0}{1}\,.
		\]
	\end{lemma}
	\vskip0.2cm 
	In particular, using this explicit formula, we extend the flow $\mathcal{S}^-(t)$ continuously from $\Htwo$ to $H^s_+(\T)$\,, for $0\leq s\leq\frac32\;$. Therefore, we have :
	
	\begin{theorem*}\ref{th defocusing case} \textbf{.}
		The Calogero--Sutherland DNLS defocusing equation \eqref{CS-defocusing} is globally well--posed in
		$H^s_+(\T)$ for any $s\geq 0\,,$ in the sense of Remark~\ref{Rq GWP}\,. 
		In addition,
		for all $u_0\in H^s_+(\T)\,,$
		\[
		u(t,z)=\ps{(\Id-z\eee^{-it}\eee^{-2it\tilde{L}_{u_0}}S^*)^{-1}\,u_0}{1}\,, 
		\]
		is solution of the \eqref{CS-defocusing}--defocusing equation.
		Furthermore, the trajectories 
		\[
		\lracc{\mathcal{S}^-(t)u_0\,;\,t\in\R}
		\]
		are relatively compact in $ H^s_+(\T)\,.$
	\end{theorem*}

	%%%%%%%%%%%%%%%%%%%%%%%%%%%%%%%%%%%%%%%%%%%%%%
	%%%%%%%%%%%%%%%%%%%%%%%%%%%%%%%%%%%%%%%%%%%%%%
	%Finals remarks and open Problem
	%%%%%%%%%%%%%%%%%%%%%%%%%%%%%%%%%%%%%%%%%%%%%%
	%%%%%%%%%%%%%%%%%%%%%%%%%%%%%%%%%%%%%%%%%%%%%%
	\vskip0.5cm
	\section{Final remarks and open problems}
	\label{Remarks}
	\vskip0.3cm
	Let us briefly discuss here some remarks related to the previous sections.
	\vskip0.3cm
	1. One interesting feature about the focusing Calogero--Sutherland DNLS equation is that it admits a rich dynamic in comparison to the defocusing equation. For instance, as we shall see \cite{Ba23}\,, the focusing equation has a wider collection of traveling wave solutions. 
	% due to some spectral properties about the Lax operator $L_u\,$ that are not the same for $$

	% that, both focusing and defocusing equations admit traveling wave solutions with arbitrary small and large $L^2$--norms belonging to $[0,+\infty[\,.$ For more details, we refer to the upcoming work \cite{Ba}. Note that this phenomenon is uncommon for dispersive PDEs equation, where usually the traveling waves of these equations admit a lower bound on their $L^2$ norms.
	% % $L^2$--threshold below which  traveling waves cannot exist. 
	% In particular, this is the case for the Calogero--Moser DNLS equation (the version of the \eqref{CS}--equation on the real line $x\in\R$)\,, where G\'erard--Lenzmann proved in \cite[Section 3]{GL22} that the traveling waves are ground states for the energy functional, and are all of $L^2$--mass equal to $\sqrt{2\pi}\,.$
	% % \vskip0.5cm
	% % 2. There exists special solution of \eqref{CS+} of norms $\|u\|_{L^2}\geq 1$ such that they are bounded in all $H^s_+(\T)$ for all $s\geq 0\,.$ For this we refer to [??]
	\vskip0.5cm
	2. The problem of global well--posedness
	of the focusing Calogero–Sutherland DNLS equation \eqref{CS+} without restriction on the initial data is wide open. Nevertheless,  we expect that the explicit solution (equation~\eqref{dynamical explicit formula L2})
	\be\label{u=formule inversion derniere section}
	u(t,z)
	=\, \ps{(\Id-z\eee^{-it}\eee^{-2itL_{u_0}}S^*)^{-1}\,u_0}{1}\,, 
	\ee
	is a key ingredient to answer this question. Indeed, writing for all $t\in\R$\,, $\,\Sigma_t$ the operator $S\eee^{2it L_{u_0}}\eee^{it}$\,, we have by \eqref{u=formule inversion derniere section} 
	\be\label{coordonnée en sigma n,t}
	u(t,z)
	= \, \sum_{n\geq 0}\ps{u_0}{\,\Sigma_t^n1}z^n\,.
	\ee
	Observe that, if $u_0$ belongs to the space $\mathcal{J}$ generated by the orthonormal family $\lracc{\Sigma_t^n1\,,\, n\geq 0}\,,$ then using  Parseval's identity on  \eqref{coordonnée en sigma n,t}\,, we infer
	\be\label{L2 norm u(t) =u0}
	\|u(t)\|_{L^2(\T)}=\|u_0\|_{L^2}\,, \qquad t\in\R\,,
	\ee
	leading to say that the set
	$
	\lracc{u(t)\,, t\in\R}
	$
	is relatively compact in $\Ltwo\,$. 
	Hence, the integer $N_\eta$ set out in inequality~\eqref{perturbation} is now independent of $t$\,, and thus applying  inequality~\eqref{perturbation} to \eqref{control H12 de la solution}\,, we obtain for all $\eta>0\,,$
	\[
	\|(L_{u(t)}+\la)^\frac12 f\|_{L^2}^2
	\geq (1-2\eta^2)\|f\|_{\dot{H}^\frac12}^2
	+(\la-2\eta^2-2N^2_{\eta}\|u_0\|_{L^2})\|f\|_{L^2}^2\,,
	\]
	instead of having inequality~\eqref{inequality leading to assumption on the initail data} :
	\[
	\|(L_{u(t)}+\la)^\frac12 f\|_{L^2}^2
	\geq (1-\|u\|_{L^2}^2)\|f\|_{\dot{H}^\frac12}^2
	+(\la-\|u\|_{L^2}^2)\|f\|_{L^2}^2\,.
	\]
	Therefore, using \eqref{perturbation}\,, we 
	control the growth of all the Sobolev norm $\|u(t)\|_{H^s}$ for all $s\geq0 \,$, and we infer the global well--posedness of the focusing \eqref{CS+} in all $H^s_+(\T)$\,, $s>\frac32$ for arbitrary initial data. 
	In addition,  by the same manner, we deduce also  $H^\frac12$--bounds on the eigenfunctions $(f_n^{\,\eps,t})$ --inequality~\eqref{borne H12}-- implying that the flow $\mathcal{S}^+(t)$ can  be extended to $L^2_+(\T)$\,, for arbitrary initial data.
	% and the $H^{\frac12}$--Sobolev norm of $\|f_n\|_{H^\frac12}^2$ in inequality ??
	Besides, if $u_0$ does not belong to $\mathcal{J}$\,, then we expect blow--up results in finite time $T\,.$

	\vskip1cm

	%%%%%%%%%%%%%%%%%%%%%%%%%%%%%%%%%%%%%%%%%%%%%%%%%%%%%%%%
	%%%%%%%%%%%%%%%%%%%%%% Bibliographie %%%%%%%%%%%%%%%%%%%
	%%%%%%%%%%%%%%%%%%%%%%%%%%%%%%%%%%%%%%%%%%%%%%%%%%%%%%%%

\end{document}